\documentclass[a4paper]{cedram-smai-jcm}

\usepackage{enumerate}

\usepackage{mathtools, amssymb, amsthm}
\usepackage{bm, dsfont} 

\usepackage{color}

\numberwithin{equation}{section}
\numberwithin{table}{section}
\numberwithin{figure}{section}

\DeclareMathOperator*{\spn}{span}
\DeclareMathOperator*{\diag}{diag}

\DeclareMathOperator*{\rank}{rank}
\DeclareMathOperator*{\argmin}{arg\,min}

\newcommand{\transp}{^{\mathsf T}}
\newcommand{\herm}{^{\ast}}
\newcommand{\inv}{^{-1}}
\newcommand{\Id}{\mathrm{Id}}
\newcommand{\Hmix}{H_{\mathrm{mix}}}

\title[Least Squares Approximation for Noisy Samples]{Error Guarantees for Least Squares Approximation\\ with Noisy Samples in Domain Adaptation}

\author[F.~Bartel]{\firstname{Felix} \lastname{Bartel}}
\address{Chemnitz University of Technology, Faculty of Mathematics, 09107 Chemnitz, Germany}
\email{felix.bartel@mathematik.tu-chemnitz.de}
\thanks{The author was supported by the Deutscher Akademischer Austauschdienst (DAAD)}

\begin{document} 

\begin{abstract}
    Given $n$ samples of a function $f\colon D\to\mathds C$ in random points drawn with respect to a measure $\varrho_S$  we develop theoretical analysis of the $L_2(D, \varrho_T)$-approximation error.
    For a parituclar choice of $\varrho_S$ depending on $\varrho_T$, it is known that the weighted least squares method from finite dimensional function spaces $V_m$, $\dim(V_m) = m < \infty$ has the same error as the best approximation in $V_m$ up to a multiplicative constant when given exact samples with logarithmic oversampling.
    If the source measure $\varrho_S$ and the target measure $\varrho_T$ differ we are in the domain adaptation setting, a subfield of transfer learning.
    We model the resulting deterioration of the error in our bounds.

    Further, for noisy samples, our bounds describe the bias-variance trade off depending on the dimension $m$ of the approximation space $V_m$.
    All results hold with high probability.
    
    For demonstration, we consider functions defined on the $d$-dimensional cube given in unifom random samples.
    We analyze polynomials, the half-period cosine, and a bounded orthonormal basis of the non-periodic Sobolev space $H_{\mathrm{mix}}^2$.
    Overcoming numerical issues of this $H_{\text{mix}}^2$ basis, this gives a novel stable approximation method with quadratic error decay.
    Numerical experiments indicate the applicability of our results.
\end{abstract}

\maketitle

\section{Introduction}
In this paper we study the reconstruction of complex-valued functions on a $d$-dimensional domain $D\subset \mathds R^d$ from possibly noisy function values
\begin{align*}
    \bm y = \bm f + \bm \varepsilon = (f(\bm x^1) + \varepsilon_1, \dots, f(\bm x^n)+\varepsilon_n)\transp\,,
\end{align*}
which are sampled in random points $\bm x^1, \dots, \bm x^n \in D$.
We consider error bounds for the weighted least squares method for individual functions, which is common in, e.g.\ partial differential equations \cite{CCMNT15} or uncertainty quantification \cite{HD15}.
In this setting, the samples are drawn after the function is fixed in contrast to worst-case or minmax-bounds, which hold for a class of functions and usually do not include noise in the samples.
For individual function approximation the majority of $L_2$-error bounds are stated in expectation, cf.\ \cite[Thm.~1.1]{Baraud02} for penalized least-squares, \cite[Thm.~3]{CDL13} for plain least-squares or, \cite[Thm.~4.1]{HNP22}, and \cite[Thm.~6.1]{KUV21} for weighted least squares.
Bounds, which hold with high probability, are known for polynomial approximation, cf.\ \cite[Thm.~3]{MNST14}, wavelet approximation, cf.\ \cite[Thms.~3.20 \& 3.21]{LPU21}, or in a more general setting incuding noise in \cite[Thm.~4.3]{CM17} with the coarser $L_\infty$-norm instead of the natural $L_2$-norm in the estimate.
Further, in \cite[Thm.~4.1]{CM17} an error bound with the natural $L_2$-norm estimate is presented in expectation with the same behaviour as we will present with high probability.
The contribution and novelty of this paper is twofold:
\begin{itemize}
\item
    We use concentration inequalities to show error bounds in the $L_2$- and $L_\infty$-norm which hold with high probability, including the noisy case.
    The behaviour of our bound is similar to \cite[Thm.~4.1]{CM17}, which is stated in expectation.
    Approximating from an $m$-dimensional function space we achieve the best error up to a multiplicative constant using logarithmic oversampling.
    Note, there exists a distribution such that linear oversampling achieves the optimal error but this is not constructive, cf.\ \cite{DC22}.
    Including noise, our bounds reflect the typical bias-variance trade off which one wants to balance to prevent over- or underfitting.
    The results enable to give performance guarantees for model selection strategies like the balancing principle \cite{PL13, LMP20} or cross-validation \cite{BaHiPo19, BH21}.
\item
    For an application we have a look at approximation on the $d$-dimensional unit cube $[0,1]^d$ when samples are distributed uniform according to the Lebesgue measure.
    A result with focus on polynomial approximation in the one-dimensional space is \cite[Thm.~3]{MNST14} which is improved by the general result \cite[Thm.~2.1]{CM17}.
    There, the aproximation error is estimated by the $L_\infty$-error of the projection with high probability and to the more natural $L_2$-error of the projection in expectation.
    We obtain a bound by the $L_2$-error of the projection which also holds with high probability.
    A drawback of polynomials is the need for quadratic oversampling, which we show for the Legendre polynomials but holds in general, cf.\ \cite{MNST14}.
    To circumvent this, we use the eigenfunctions of the embedding $\Id \colon H^s \to L_2$ from the Sobolev space $H^s$ for $s=1, 2$ which allow for logarithmic oversampling.
    The $H^1$ basis, also known as half-period cosine, was introduced in \cite{Krein35} and has become the standard in many applications and is researched thouroughly, cf.\ \cite{IN08, WW08, Adcock10, AH11, DDP14, SNC16, CKNS16, KMNN21}.
    But also for functions in Sobolev spaces $H^s$ of higher smoothness their convergence is limited to be linear in theory (the rate $3/2$ can be observed in practice).
    This can be improved by using the $H^2$ basis, examined theoretically in \cite[Section 3]{AIN12} to have quadratic convergence.
    So far it is not used as it is prone to numerical errors and unusable for higher degree approximation.
    Here, we propose an approximation and prove its accuracy which leads to a numerically stable way for approximating non-periodic uniform data with quadratic convergence.
\end{itemize}

For a more detailed formulation we need some notation.
Given an $m$-dimensional function space $V_m \subset L_2$, we define the best possible approximation (projection) to $f\colon D\to\mathds C$ in $V_m$ and its error:
\begin{align*}
    P(f, V_m, L_p)
    = \argmin_{g\in V_m} \|f-g\|_{L_p}
    \text{ and }
    e(f, V_m, L_p)_{L_q}
    = \|f-P(f, V_m, L_p)\|_{L_q}
\end{align*}
for $p,q\in\{2,\infty\}$.
Note, since $V_m$ is finite-dimensional the minimum is actually attained.
Following \cite{Baraud02, CDL13, MNST14, CM17, KUV21, LPU21}, we use weighted least squares $S_m$, defined in \eqref{eq:lsqrmatrix}, as underlying approximation method.
Because of its linearity, the approximation error $\|f-S_m\bm y\|_{L_2}$ splits as follows:
\begin{align*}
    \|f-S_m\bm y\|_{L_2}^2
    &= e(f,V_m,L_2)_{L_2}^2 + \|P(f,V_m,L_2)-S_m\bm y\|_{L_2}^2 \\
    &\le
    \underbrace{
        e(f,V_m,L_2)_{L_2}^2
    }_{\text{truncation error}}
    +2\underbrace{
        \|P(f,V_m,L_2)-S_m\bm f\|_{L_2}^2
    }_{\text{discretization error}}
    +2\underbrace{
        \|S_m\bm\varepsilon\|_{L_2}^2
    }_{\text{noise error}} \,.
\end{align*}
For fixed number of points $n$, we have a look at the behaviour with respect to $m$, the dimension of the approximation space $V_m$.
The truncation error is the best possible benchmark and usually has polynomial decay $m^{-s}$ for some rate $s\ge 1$ depending on $f$ and the choice of $V_m$.
We show, that the discretization error obeys the same decay as the truncation error.
Thus, given logarithmic oversampling, we obtain the best possible error up to a multiplicative constant in the noiseless case, cf.\ Theorem~\ref{L2wo}.

Including noise, we show that we get an additional summand growing linear in $m$, cf.\ Theorems~\ref{L2w}.
This resambles the well-known bias-variance trade off modeling the over- and undersmoothing effects which one wants to balance, cf.\ \cite{GKKW02, PL13}.
This linear behaviour in $m$ is approved by \cite[Thm.~4.9]{LMP20} (by using the regularization $g_\lambda(\sigma) = 1/(\lambda+\sigma)$ with $\lambda = 0$).
An example of that behaviour for $D = [0,1]$ and $\varrho_T = \mathrm dx$ being the Lebesgue measure is depicted in Figure~\ref{fig:experiment_1d} where the detailed example is found in Section~\ref{sec:application}.
\begin{figure}
    \centering
    \begingroup
  \makeatletter
  \providecommand\color[2][]{\GenericError{(gnuplot) \space\space\space\@spaces}{Package color not loaded in conjunction with
      terminal option `colourtext'}{See the gnuplot documentation for explanation.}{Either use 'blacktext' in gnuplot or load the package
      color.sty in LaTeX.}\renewcommand\color[2][]{}}\providecommand\includegraphics[2][]{\GenericError{(gnuplot) \space\space\space\@spaces}{Package graphicx or graphics not loaded}{See the gnuplot documentation for explanation.}{The gnuplot epslatex terminal needs graphicx.sty or graphics.sty.}\renewcommand\includegraphics[2][]{}}\providecommand\rotatebox[2]{#2}\@ifundefined{ifGPcolor}{\newif\ifGPcolor
    \GPcolortrue
  }{}\@ifundefined{ifGPblacktext}{\newif\ifGPblacktext
    \GPblacktexttrue
  }{}\let\gplgaddtomacro\g@addto@macro
\gdef\gplbacktext{}\gdef\gplfronttext{}\makeatother
  \ifGPblacktext
\def\colorrgb#1{}\def\colorgray#1{}\else
\ifGPcolor
      \def\colorrgb#1{\color[rgb]{#1}}\def\colorgray#1{\color[gray]{#1}}\expandafter\def\csname LTw\endcsname{\color{white}}\expandafter\def\csname LTb\endcsname{\color{black}}\expandafter\def\csname LTa\endcsname{\color{black}}\expandafter\def\csname LT0\endcsname{\color[rgb]{1,0,0}}\expandafter\def\csname LT1\endcsname{\color[rgb]{0,1,0}}\expandafter\def\csname LT2\endcsname{\color[rgb]{0,0,1}}\expandafter\def\csname LT3\endcsname{\color[rgb]{1,0,1}}\expandafter\def\csname LT4\endcsname{\color[rgb]{0,1,1}}\expandafter\def\csname LT5\endcsname{\color[rgb]{1,1,0}}\expandafter\def\csname LT6\endcsname{\color[rgb]{0,0,0}}\expandafter\def\csname LT7\endcsname{\color[rgb]{1,0.3,0}}\expandafter\def\csname LT8\endcsname{\color[rgb]{0.5,0.5,0.5}}\else
\def\colorrgb#1{\color{black}}\def\colorgray#1{\color[gray]{#1}}\expandafter\def\csname LTw\endcsname{\color{white}}\expandafter\def\csname LTb\endcsname{\color{black}}\expandafter\def\csname LTa\endcsname{\color{black}}\expandafter\def\csname LT0\endcsname{\color{black}}\expandafter\def\csname LT1\endcsname{\color{black}}\expandafter\def\csname LT2\endcsname{\color{black}}\expandafter\def\csname LT3\endcsname{\color{black}}\expandafter\def\csname LT4\endcsname{\color{black}}\expandafter\def\csname LT5\endcsname{\color{black}}\expandafter\def\csname LT6\endcsname{\color{black}}\expandafter\def\csname LT7\endcsname{\color{black}}\expandafter\def\csname LT8\endcsname{\color{black}}\fi
  \fi
    \setlength{\unitlength}{0.0500bp}\ifx\gptboxheight\undefined \newlength{\gptboxheight}\newlength{\gptboxwidth}\newsavebox{\gptboxtext}\fi \setlength{\fboxrule}{0.5pt}\setlength{\fboxsep}{1pt}\definecolor{tbcol}{rgb}{1,1,1}\begin{picture}(6800.00,1980.00)\gplgaddtomacro\gplbacktext{\csname LTb\endcsname \put(270,408){\makebox(0,0)[r]{\strut{}$0$}}\csname LTb\endcsname \put(270,1449){\makebox(0,0)[r]{\strut{}$1$}}\csname LTb\endcsname \put(382,204){\makebox(0,0){\strut{}$0$}}\csname LTb\endcsname \put(1424,204){\makebox(0,0){\strut{}$1$}}}\gplgaddtomacro\gplfronttext{\csname LTb\endcsname \put(903,1755){\makebox(0,0){\strut{}$m=2$ (underfitting)}}}\gplgaddtomacro\gplbacktext{\csname LTb\endcsname \put(1965,408){\makebox(0,0)[r]{\strut{}$0$}}\csname LTb\endcsname \put(1965,1449){\makebox(0,0)[r]{\strut{}$1$}}\csname LTb\endcsname \put(2077,204){\makebox(0,0){\strut{}$0$}}\csname LTb\endcsname \put(3119,204){\makebox(0,0){\strut{}$1$}}}\gplgaddtomacro\gplfronttext{\csname LTb\endcsname \put(2598,1755){\makebox(0,0){\strut{}$m=4$}}}\gplgaddtomacro\gplbacktext{\csname LTb\endcsname \put(3660,408){\makebox(0,0)[r]{\strut{}$0$}}\csname LTb\endcsname \put(3660,1449){\makebox(0,0)[r]{\strut{}$1$}}\csname LTb\endcsname \put(3772,204){\makebox(0,0){\strut{}$0$}}\csname LTb\endcsname \put(4813,204){\makebox(0,0){\strut{}$1$}}}\gplgaddtomacro\gplfronttext{\csname LTb\endcsname \put(4292,1755){\makebox(0,0){\strut{}$m=16$ (overfitting)}}}\gplgaddtomacro\gplbacktext{\csname LTb\endcsname \put(5411,408){\makebox(0,0)[r]{\strut{}}}\csname LTb\endcsname \put(5411,929){\makebox(0,0)[r]{\strut{}}}\csname LTb\endcsname \put(5411,1449){\makebox(0,0)[r]{\strut{}}}\csname LTb\endcsname \put(5467,204){\makebox(0,0){\strut{}}}\csname LTb\endcsname \put(5988,204){\makebox(0,0){\strut{}}}\csname LTb\endcsname \put(6508,204){\makebox(0,0){\strut{}}}}\gplgaddtomacro\gplfronttext{\csname LTb\endcsname \put(5987,223){\makebox(0,0){\strut{}$m$}}\csname LTb\endcsname \put(5987,1755){\makebox(0,0){\strut{}$\|f-S_m\bm y\|_{\mathrm L_2}^2$}}}\gplbacktext
    \put(0,0){\includegraphics[width={340.00bp},height={99.00bp}]{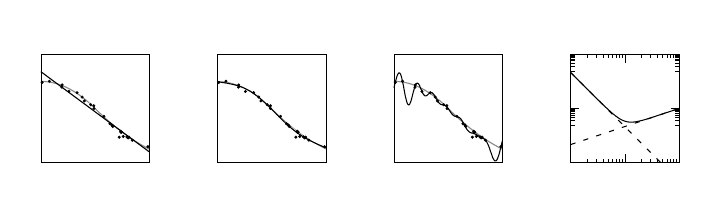}}\gplfronttext
  \end{picture}\endgroup
     \caption{One-dimensional approximation on the unit-interval for three different choices of $m$ and the schematic behaviour of the $L_2$-approximation error $\|f-S_m\bm y\|_{L_2}^2$ (solid line) split into the error for exact function values $\|f-S_m\bm f\|_{L_2}^2$ and the noise error $\|S_m\bm\varepsilon\|_{L_2}^2$ (dashed lines) with respect to $m$.}\label{fig:example_1d}
\end{figure}
Our central theorem, complying this behaviour, looks as follows:

\begin{theorem}\label{L2w} Let $f\colon D\to\mathds C$, $\bm x^1,\dots,\bm x^n$, $n\in\mathds N$ be points drawn according to a probability measure $\mathrm d\varrho_S = 1/\beta\,\mathrm d\varrho_T$ and $\bm y = \bm f + \bm\varepsilon = (f(\bm x^1)+\varepsilon_1, \dots, f(\bm x^n)+\varepsilon_n)\transp$ noisy function values where
    $\bm\varepsilon = (\varepsilon_1,\dots,\varepsilon_n)\transp$ is a vector of independent complex-valued mean-zero random variables satisfying $\mathds E(\lvert\varepsilon_i\rvert^2) \le \sigma^2$ and $\lvert\varepsilon_i\rvert \le B$ for $i=1,\dots,n$.
    Let further, $t\ge 0$, $V_m$ be an $m$-dimensional function space with an $L_2(D, \varrho_T)$-orthonormal basis $\eta_0, \dots, \eta_{m-1}$ satisfying
    \begin{align*}
        10\| \beta(\cdot) N(V_m, \cdot) \|_\infty (\log(m)+t) \le n
        \quad\text{with}\quad
        N(V_m, \cdot) = \sum_{k=0}^{m-1} \lvert\eta_k(\cdot)\rvert^2 \,.
    \end{align*}
    Then, for $S_m$ the weighted least squares method defined in \eqref{eq:lsqrmatrix} with $\omega_i = \beta(\bm x^i)$, we have with joint probability exceeding $1-3\exp(-t)$:
    \begin{align*}
        \| f - S_m \bm y \|_{L_2}^{2}
        &\le 14 \Big( e(f,V_m,L_2)_{L_2}+ \sqrt{\frac t n} e(f,V_m,L_2)_{L_\infty} \Big)^2 \\
        &\quad + 
        4\|\beta\|_\infty \Big(
        \frac{m}{n}
        \Big( 14 B \sqrt{t\sigma^2} +\sigma^2 \Big)
        +
        \frac{128 B^2 t}{n}
        \Big) \,,
    \end{align*}
    where $L_2 = L_2(D, \varrho_T)$ and $L_\infty = L_\infty(D, \varrho_T)$.
\end{theorem} 

The first line of the bound corresponds to the truncation error and discretization error, decaying in $m$.
Note, that the $L_\infty$-term with the prefactor $n^{-1/2}$ behaves as the $L_2$-term whenever $\beta$ is bounded from below, cf.\ Theorem~\ref{L2wo}.
The second line is the error due to noise, increasing in $m$, cf.\ Figure~\ref{fig:example_1d}.
The estimation of the noise error is using a Hanson-Wright concentration inequality, which can be found using different assumptions.
Thus, we can replace the noise model by general Bernstein conditions, cf.\ Lemma~\ref{hanson-wright}, or sub-Gaussian noise, cf.\ \cite{RV13}.
This theorem extends to the $L_\infty$ case:

\begin{theorem}[$L_\infty$-error bound with noise]\label{Linfw}Let the assumptions of Theorem~\ref{L2w} hold.
    Then, for $S_m$ the weighted least squares method defined in \eqref{eq:lsqrmatrix} with $\omega_i = \beta(\bm x^i)$, we have with probability exceeding $1-3\exp(-t)$:
    \begin{align*}
        \| f - S_m \bm y \|_{ L_\infty}
        &\le \Big(1+\sqrt{5N(V_m)}\Big)
        \Big( e(f,V_m, L_\infty)_{ L_\infty}+ \sqrt{\frac t n} e(f,V_m, L_\infty)_{ L_2} \Big) \\
        &\quad +
        2\sqrt{\|\beta\|_\infty N(V_m)}
        \sqrt{
        \frac{m}{n}
        \Big( 14 B \sqrt{t\sigma^2} +\sigma^2 \Big)
        +
        \frac{128 B^2 t}{n}
        } \,.
    \end{align*}
\end{theorem} 

The bound is similar to \cite[Thm.~3.21]{LPU21} in the wavelet setting but we use the best approximation with respect to the more natural $L_\infty$ instead of $L_2$.
In addition to the error of the best approximation we now have the additional factor $N(V_m)$ due to using the norm estimate $\|g\|_{L_\infty} \le \sqrt{N(V_m)}\|g\|_{L_2}$ for functions $g\in V_m$.
The same factor appears when approximating the worst-case error where it is known to be necessary in various examples, e.g.\ \cite[Sec.~7]{PU21} or \cite[Thm~.1.1]{Temlyakov93}.

The sampling measure $ \varrho_S(E) \coloneqq \int_E 1/\beta \;d\varrho_T$, induced by the probability distribution $\beta$, may differ from the error measure $\varrho_T$, which is known as the \emph{change of measure} and has applications in domain adaptation, cf.~\cite{PY10}.
We assume to know $\beta$ exactly but it may be approximated as well, cf.\ \cite{GMMNPPSZ22}.
Note, $\beta$ affects the maximal size of $V_m$ in the assumption and the amplification of the noise in bound.
There are two extremal cases:
\begin{enumerate}[(i)]
\item
    Having $\beta(\bm x) = m/N(V_m, \bm x)$, as it was done in \cite{HD15, NJZ17, CM17, KUV21}, we obtain the assumption
    \begin{align*}
        10\| \beta(\cdot) N(V_m, \cdot) \|_\infty (\log(m)+t)
        = 10m (\log(m)+t)
        \le n \,,
    \end{align*}
    which allows for the biggest choice of $m$ possible.
    But this spoils $\|\beta\|_\infty$ in the error bound when the Christoffel function attains small values.
\item 
    For domains $D$ with bounded measure, we may choose $\beta(\bm x) = \varrho_T(D)$, as it was done in \cite{CDL13, CM17, LPU21}.
    As all weights $\omega_i = \varrho_T(D)$, $S_m$ becomes the plain least squares method.
    In this case, $\|\beta\|_\infty$ is minimal and noise is amplified the least.
    But this choice spoils the assumption on the choice of $m$ when the Christoffel function $N(V_m, \bm x)$ attains big values.
    This effect is controllable, for instance, when working with a bounded orthonormal system (BOS) ($\|\eta_k\|_\infty \le B$ for some $B>0$ and all $k$).
    Then
    \begin{align*}
        N(V_m)
        \le \sum_{k=0}^{m-1} \|\eta_k\|_\infty^2
        \le m B^2
    \end{align*}
    and the assumption on the size of $V_m$ can be replaced by
    \begin{align*}
        10\| \beta(\cdot) N(V_m, \cdot) \|_\infty (\log(m)+t)
        \le 10\varrho_T(D)Bm (\log(m)+t)
        \le n \,.
    \end{align*}
\end{enumerate}

An interesting example, where these effects occur, is the approximation of functions on the unit interval $D=[0,1]$ from samples given in uniformly random points.
\begin{itemize}
\item
    When using algebraic polynomials and the Lebesgue error measure $\mathrm d\varrho_T=\mathrm dx$ we have to choose $\beta \equiv 1$ to obtain uniform random points.
    Orthogonalizing algebraic polynomials with respect to the Lebesgue measure, we obtain $\eta_k = P_k/\|P_k\|_{L_2((0,1),\mathrm dx)}$ Legendre polynomials for our approximation space $V_m$.
    Since $\|P_k\|_{L_2((0,1),\mathrm dx)}^2 = 2k+1$ and $P_k(0)=1$, we have
    \begin{align}\label{eq:stupidlegendre}
        N(V_m, 0)
        = \sum_{k=0}^{m-1} \frac{\lvert P_k(0)\rvert^2}{\|P_k\|_{L_2}^2}
        = \sum_{k=0}^{m-1} (2k+1)
        = m^2\,.
    \end{align}
    Thus, this case falls into category (i) from above and spoils our choice of $m \le \sqrt{n}$, i.e., quadratic oversampling as in \cite{MNST14}.
\item
    When using algebraic polynomials and the Chebyshev error measure $\mathrm d\varrho_T=(1-(2x-1)^2)^{-1/2}\;\mathrm dx$ we have to choose $\beta(x) = \frac{\pi}{4} (1-(2x-1)^2)^{-1/2}$ to obtain uniform random points.
    Orthogonalizing algebraic polynomials with respect to the Chebyshev measure, we obtain Chebyshev polynomials $\eta_k(x) = T_k(x) = \cos(k\arccos(2x-1))$ for our approximation space $V_m$.
    These are a BOS, but the distribution $\beta$ spoils both the assumption on $m$ and the error bound, since $\beta$ diverges at the border (this effect can be circumvented using a padding technique at the border as we discuss in Section~\ref{sec:application}).
\item
    In Section~\ref{subs:hs} we construct orthogonal functions with respect to the Sobolev space inner product
    \begin{align*}
        \langle f, g\rangle_{H^s(0,1)}
        = \langle f, g\rangle_{L_2((0,1), \mathrm dx)}
        + \langle f^{(s)}, g^{(s)}\rangle_{L_2((0,1), \mathrm dx)} \,,
    \end{align*}
    for $s=1$ and $s=2$, which are orthogonal with respect to the $L_2((0,1), \mathrm dx)$ inner product as well.
    For $s=1$, these functions were originally introduced in \cite{Krein35} and for $s=2$ in \cite[Section 3]{AIN12}, where also higher orders can be found.

    We show that they form a BOS and, by (ii) above, this basis is then suitable for approximation in uniform random points on $D=[0,1]$ using plain least squares and only logarithmic oversampling.
    The $H^2$ basis is prone to numerical errors.
    To overcome this, we propose a numerically stable approximation and proof its accuracy.
\end{itemize}

As for the structure of this paper, we start with some tools from probability theory in Section~\ref{sec:prob}.
In Section~\ref{sec:weighted} we show error bounds for the weighted least squares method.
The construction of the $H^1$ and $H^2$ basis mentioned above are found in Section~\ref{sec:application} along with a comparison to the Legendre and Chebyshev polynomials.
To indicate the applicability of our error bounds and the proposed basis, we conduct numerical experiments in one and five dimensions.

 \section{Tools from probabilty theory}\label{sec:prob}

For the later analysis we need concentration inequalities starting with Bernstein's inequality, which is found in the standard literature, cf.\ \cite[Theorem 6.12]{SC08} or \cite[Corollary 7.31]{FR13}.

\begin{theorem}[Bernstein] Let $\xi_1,\dots,\xi_n$ be independent real-valued mean-zero random variables satisfying $\mathds E(\xi_i^2) \le \sigma^2$ and $\|\xi_i\|_\infty \le B$ for $i=1,\dots,n$ and real numbers $\sigma^2$ and $B$.
    Then
    \begin{align*}
        \frac 1n \sum_{i=1}^n \xi_i
        \le \frac{2B t}{3n} + \sqrt{\frac{2\sigma^2 t}{n}}
    \end{align*}
    with probability exceeding $1-\exp(-t)$.
\end{theorem} 

Bernstein's inequality gives a concentration bound for the sum of independent random variables.
We need similar bounds for quadratic forms in random vectors, which are known as Hanson-Wright inequalities.
To formulate them, we need to introduce the spectral norm and the Frobenius norm of a matrix $\bm A\in\mathds C^{m\times n}$
\begin{align*}
    \|\bm A\|_{2\to 2} = \sqrt{\lambda_{\max}(\bm A\herm \bm A)} = \sigma_{\max}(\bm A)
    \quad\text{and}\quad
    \|\bm A\|_F = \sqrt{\sum_{k=1}^{m}\sum_{i=1}^{n} \lvert a_{k,i}\rvert^2 } \,,
\end{align*}
where $\lambda_{\max}$ and $\sigma_{\max}$ denote the largest eigenvalues and singular values, respectively.
The following result is such an inequality with a Bernstein condition on the random variables and was shown in \cite[Theorem 3]{Bellec19}.

\begin{theorem}[Hanson-Wright]\label{bellec}Let $\bm\xi = (\xi_1,\dots,\xi_n)\transp$ be a vector of independent mean-zero random variables such that for all integers $p\ge 1$
    \begin{align}\label{bernsteincondition}
        \mathds E(\lvert\xi_i\rvert^{2p})
        \le p! B^{2p-2}\sigma_i^2/2
    \end{align}
    for real numbers $B\ge 0$, $\sigma_i \ge 0$, and all $i = 1,\dots,n$.
    Let further $\bm A \in\mathds C^{n\times n}$ and $m = \mathds E(\bm\xi\herm\bm A\bm\xi)$.
    Then $\bm D_\sigma = \diag(\sigma_1, \dots, \sigma_n)$
    \begin{align*}
        \bm\xi\herm\bm A\bm\xi - m
        \le 256 B^2 \|\bm A \bm D_\sigma\|_{2\to 2} t+8\sqrt 3 B \|\bm A \bm D_\sigma\|_F\sqrt{t}
    \end{align*}
    with probability exceeding $1-\exp(-t)$.
\end{theorem} 

The following is a special case of the above Hanson-Wright inequality for Hermitian positive semi-definite matrices and random variables with known variance $\mathds E(\lvert\xi_i\rvert^2)$ and a uniform bound $\|\xi_i\|_\infty$.

\begin{corollary}\label{hanson-wright} Let $\bm\xi = (\xi_1,\dots,\xi_n)\transp$ be a vector of independent complex-valued mean-zero random variables satisfying $\mathds E(\lvert\xi_i\rvert^2) \le \sigma^2$ and $\|\xi_i\|_\infty \le B$ for $i=1,\dots,n$.
    Then for all $\bm A \in\mathds C^{m\times n}$
    \begin{align*}
        \|\bm A\bm\xi\|_2^2
        \le 128 B^2 \|\bm A\|_{2\to 2}^2 t+(8\sqrt 3 B\sqrt{t\sigma^2} + \sigma^2 )\|\bm A\|_F^2
    \end{align*}
    with probability exceeding $1-\exp(-t)$.
\end{corollary} 

\begin{proof} Since $\|\bm A\bm\xi\|_2^2 = \bm\xi\herm\bm A\herm\bm A\bm\xi$ is a quadratic form we want to apply Theorem~\ref{bellec} on $\bm A\herm\bm A$.
    For that we check the moment condition \eqref{bernsteincondition} on $\xi_1^2, \dots, \xi_n^2$.
    For $p=1$ it is fulfilled for constants $B/\sqrt 2$ and $(\sqrt 2\sigma_i)^2$.
    For $p\ge 2$, we have $p! \ge 2^{p-1}$ and obtain
    \begin{align*}
        \mathds E(\lvert a_i\xi_i\rvert^{2p})
        &\le \|\xi_i\|_\infty^{2p-2}\mathds E(\lvert\xi_i\rvert^2) \\
        &\le (B)^{2p-2} \sigma^2 \\
        &\le p! \Big(\frac{B}{\sqrt{2}}\Big)^{2p-2} \frac{(\sqrt 2\sigma)^2}{2} \,.
    \end{align*}
    Therefore, Theorem~\ref{bellec} is applicable.

    It is left to estimate the expected value.
    Since $\xi_1,\dots,\xi_n$ are independent and have bounded variance, we obtain
    \begin{align*}
        \mathds E\Big(\|\bm A\bm\xi\|_2^2\Big)
        &= \sum_{k=1}^{m}
        \sum_{i=1}^{n} \sum_{j=1}^{n} a_{i,k}\overline{a_{j,k}} \, \mathds E(\xi_i \overline{\xi_j}) \\
        &= \sum_{k=1}^{m}
        \Big(\sum_{i=1}^{n} \sum_{j\neq i} a_{i,k}\overline{a_{j,k}} \, \mathds E( \xi_i \overline{\xi_j})\Big)
        + \sum_{i=1}^{n} \lvert a_{i,k}\rvert^2\mathds E(\lvert\xi_i\rvert^2) \\
        &\le \sigma^2 \|\bm A\|_F^2 \,.\qedhere
    \end{align*}
\end{proof}

The following tool is a concentration bound on the maximal singular values of random matrices which was shown in \cite[Theorem~1.1]{Tr12}.

\begin{lemma}[Matrix Chernoff]\label{matrixchernoff} For a finite sequence $\bm A_1, \dots, \bm A_n \in \mathds {C}^{m \times m}$ of independent, Hermitian, positive semi-definite random matrices satisfying $\lambda_{\max}(\bm A_i) \le R$ almost surely it holds
    \begin{align*}
        \mathds P\Big(
            \lambda_{\min}\Big( \sum_{i=1}^{n} \bm A_i \Big) \le (1-t)\mu_{\min}
        \Big)
        &\le {m}\exp\Big(-\frac{\mu_{\min}}{R} (t+(1-t)\log(1-t)) \Big) \\
        &\le {m}\exp\Big(-\frac{\mu_{\min}t^2}{2R} \Big)
    \end{align*}
    and
    \begin{align*}
        \mathds P\Big(
            \lambda_{\max}\Big( \sum_{i=1}^{n} \bm A_i \Big) \ge (1+t)\mu_{\max}
        \Big)
        &\le {m}\exp\Big(-\frac{\mu_{\max}}{R} (-t+(1+t)\log(1+t)) \Big) \\
        &\le {m}\exp\Big(-\frac{\mu_{\max}t^2}{3R} \Big) 
    \end{align*}
    for $t\in[0,1]$ where $\mu_{\min} \coloneqq \lambda_{\min} (\sum_{i=1}^{n} \mathds E (\bm A_i))$ and $\mu_{\max} \coloneqq \lambda_{\max} (\sum_{i=1}^{n} \mathds E (\bm A_i))$.
\end{lemma} 

\begin{proof} The first estimates are provided by \cite[Theorem 1.1]{Tr12}. Based on the Taylor expansion
    \begin{align*}
    (1+t)\log(1+t) = t + \sum_{k=2}^{\infty} \frac{(-1)^k}{k(k-1)} t^k    \,,
    \end{align*}
    which holds true for $t\in[-1,1]$, we further derive the inequalities
    \begin{flalign*}
     && t+(1-t)\log(1-t) &= \sum_{k=2}^{\infty} \frac{1}{k(k-1)} t^k \ge \frac{t^2}{2}    && \\
    \text{and} &&    -t+(1+t)\log(1+t) &=    \sum_{k=2}^{\infty} \frac{(-1)^k}{k(k-1)} t^k \ge \frac{t^2}{2} - \frac{t^3}{6} \ge    \frac{t^2}{3} &&
    \end{flalign*}
    for the range $t\in[0,1]$.
\end{proof}

 \section{Error bounds for least squares}\label{sec:weighted}
In this section we develop $ L_2$- and $ L_\infty$-error bounds for the least squares method.
To this end we introduce some notation and the method itself.
Let $\eta_0, \dots, \eta_{m-1} \colon D\to\mathds C$ be an $L_2$-orthonormal basis of $V_m$,
\begin{align*}
    N(V_m, \bm x)
    = \sum_{k=0}^{m-1} \lvert\eta_k(\bm x)\rvert^2
    \quad\text{and}\quad
    N(V_m)
    = \sup_{\bm x\in D} N(V_m, \bm x)
\end{align*}
be the \emph{Christoffel function} and its supremum.
For our approximation method $S_m$ we use the \emph{weighted least squares approximation} depending on $\eta_0, \dots, \eta_{m-1}$ and $\bm x ^1, \dots, \bm x^n$:
\begin{align}
    (S_m \bm y)(\bm x)
    = \sum_{k=0}^{m-1}\hat g_k \eta_k(\bm x)
    \quad\text{with}\quad
    &\bm{\hat g} = \argmin_{\bm{\hat a}\in\mathds C^m} \| \bm L \bm{\hat a} - \bm y \|_{\bm W}^2 \,, \nonumber \\
    \bm L = \begin{bmatrix}\eta_0(\bm x^1) & \dots & \eta_{m-1}(\bm x^1) \\
    \vdots & \ddots & \vdots \\
    \eta_0(\bm x^n) & \dots & \eta_{m-1}(\bm x^n)\end{bmatrix} \in \mathds C^{n\times m}
    &\text{, and}\quad
    \bm W = \begin{bmatrix}\omega_1\\&\ddots\\&&\omega_n\end{bmatrix} \in \mathds [0, \infty)^{n\times n}
    \label{eq:lsqrmatrix}
\end{align}
where $\|\bm L\bm{\hat a}-\bm y\|_{\bm W}^2 = (\bm L\bm{\hat a}-\bm y)\herm\bm W(\bm L\bm{\hat a}-\bm y)$.
If all weights are equal we speak of \emph{plain least squares approximation}.

The coefficients $\bm{\hat g}$ of the approximation $S_m\bm y$ are found by solving the normal equation
\begin{align*}
    \bm{\hat g} 
    = (\bm L\herm\bm W\bm L)\inv\bm L\herm\bm W\bm y \,.
\end{align*}
Doing this by the means of an iterative solver, the stability and the iteration count for a desired precision are determined by the spectral properties of the above matrix, cf.\ \cite[Theorem~3.1.1]{Greenbaum97}.
However, these are fully determined by the spectral properties of $\bm W^{1/2}\bm L$, since for a singular value decomposition $\bm W^{1/2}\bm L = \bm U\bm\Sigma\bm V\herm$, we obtain
\begin{align}\label{eq:svd}
    (\bm L\herm\bm W\bm L)\inv\bm L\herm\bm W^{1/2}
    = \bm V (\bm\Sigma\herm\bm\Sigma)\inv \bm\Sigma\bm U\herm \,.
\end{align}
For $f = \sum_{k=0}^{m-1} \hat f_k \eta_k \in V_m$, the singular values of $\bm W^{1/2}\bm L$ relate the coefficients $\bm{\hat f} = (\hat f_0,\dots,\hat f_{m-1})\transp$ with the samples $\bm f = (f(\bm x^1), \dots, f(\bm x^n))\transp = \bm L\bm{\hat f}$.
Such connection is known as $L_2$-Marcinkiewicz-Zygmund inequality for $V_m$.
It was established in \cite[Thm.~2.1]{CM17} that random points also fulfill this, which is central in all theorems presented.
This makes it applicable in a very general setting, cf.\ \cite[Theorem~2.3]{NSU21}, \cite[Theorem~5.1]{MU21}, \cite[Lemma~2.1]{DC22a}, or \cite[Theorem~2.1]{BSU22}.

\begin{lemma}\label{l:trand} Let $t\ge 0$, $n\in\mathds N$, $\bm x^1,\dots,\bm x^n$ be points drawn according to a probability measure $\mathrm d\varrho_S = 1/\beta\,\mathrm d\varrho_T$.
    Let further, $V_m$ be an $m$-dimensional function space with an orthonormal basis $\eta_0, \dots, \eta_{m-1}$ in $ L_2$ with $m$ satisfying
    \begin{align*}
        10\| \beta(\cdot) N(V_m, \cdot) \|_\infty (\log(m)+t) \le n
    \end{align*}
    and $\bm L, \bm W$ be as in \eqref{eq:lsqrmatrix} with $\omega_i = \beta(\bm x^i)$.
    Then
    \begin{align*}
        \frac n2 \|\bm{\hat g}\|_2^2
        \le \|\bm W^{1/2}\bm L\bm{\hat g}\|_2^2
        \le \frac{3n}{2} \|\bm{\hat g}\|_2^2
        \quad\text{for all}\quad
        \bm{\hat g}\in\mathds C^m,
    \end{align*}
    where each inequality holds with probability exceeding $1-\exp(-t)$, repectively.
\end{lemma} 

The proof ideas go back to \cite[Thm.~1]{CDL13} and \cite[Thm.~2.1]{CM17} but for the sake of readability we state it here as well.

\begin{proof} The result is a direct consequence of Tropp's result in Lemma~\ref{matrixchernoff}.
    For a randomly chosen point $\bm x^i$ we define the random rank-one matrix $\bm A_i = \frac{1}{n}\beta(\bm x^i) (\bm y^i \otimes \bm y^i)$ with $\bm y^i = (\eta_0(\bm x^i), \dots, \eta_{m-1}(\bm x^i))\transp$.
    By construction, it holds
    \begin{align*}
        \sum_{i=1}^n \bm A_i
        = \bm L\herm\bm W\bm L
    \end{align*}
    and by the orthogonality of $\eta_k$
    \begin{align*}
        \Big( \mathds E ( \bm A_i) \Big)_{k,l}
        = \frac 1n \int_{D} \eta_k(\bm x) \overline{\eta_l(\bm x)} \beta(\bm x) \beta\inv(\bm x) \; d\varrho_T(\bm x)
        = \frac{\delta_{k,l}}{n} \,,
    \end{align*}
    which gives $\mathds E \Big(\sum_{i=1}^{n} \bm A_i\Big) = \Id_{m\times m}$ and, therefore, $\mu_{\max} = 
    \mu_{\min} = 1$.
    Further, we have
    \begin{align*}
            \lambda_{\max}\Big(\frac{1}{n}\beta(\bm x^i) (\bm y^i \otimes \bm y^i)\Big) 
            = \frac{1}{n}\beta(\bm x^i)\|\bm y^i\|_2^2 
            \le \frac 1 n \|\beta(\cdot) N(V_m, \cdot) \|_\infty \,.
    \end{align*}
    Lemma~\ref{matrixchernoff} with $t = 1/2$ then gives the lower bound
    \begin{align*}
            \mathds P \Big(\lambda_{\min}\Big(\frac{1}{n}\bm L\herm\bm W\bm L\Big) \le \frac 12\Big)
            &\leq {m}\exp\Big(-\frac{n}{10} \|\beta(\cdot) N(V_m,\cdot) \|_\infty\inv \Big) \,,
    \end{align*}
    which is smaller than $\exp(-t)$ by the assumption on $m$.

    The bound for the largest eigenvalue works analogue.
\end{proof}

We now formulate a bound on the $ L_2$-error of the weighted least squares method for exact function values.
This result is heavily based on \cite[Theorem~3.20]{LPU21} which extends to a more general setting.

\begin{theorem}[$ L_2$-error bound without noise]\label{L2wo} Let $f\colon D\to\mathds C$, $\bm x^1,\dots,\bm x^n$, $n\in\mathds N$ be points drawn according to a probability measure $\mathrm d\varrho_S = 1/\beta\,\mathrm d\varrho_T$ and $\bm y = (f(\bm x^1), \dots, f(\bm x^n))\transp$ exact function values.
    Let further, $t\ge 0$, $V_m$ be an $m$-dimensional function space with an orthonormal basis $\eta_0, \dots, \eta_{m-1}$ satisfying
    \begin{align*}
        10\| \beta(\cdot) N(V_m, \cdot) \|_\infty (\log(m)+t) \le n \,.
    \end{align*}
    Then, for $S_m$ the weighted least squares method defined in \eqref{eq:lsqrmatrix} with $\omega_i = \beta(\bm x^i)$, we have with probability exceeding $1-2\exp(-t)$:
    \begin{align*}
        \| f - S_m \bm y \|_{ L_2}^{2}
        &\le 8 \Big( e(f,V_m, L_2)_{ L_2}+ \sqrt{\frac t n} e(f,V_m, L_2)_{ L_\infty} \Big)^2 \,.
    \end{align*}
\end{theorem} 

\begin{proof} For abbreviation, we use $ e_2 = e(f, V_m,  L_2)_{ L_2} $ and $ e_\infty = e(f, V_m,  L_2)_{ L_\infty} $.
    Further, we define the event
    \begin{align}\label{eq:L2woA}
        A \coloneqq
        \Big\{ \bm x^1, \dots, \bm x^n \in D : 
        \frac n2
        \le \|\bm W^{1/2}\bm L \|_{2\to 2}^2
        \Big\}
    \end{align}
    which has probability $\mathds P(A) \ge 1-\exp(-t)$ by Lemma~\ref{l:trand} and the assumption on $V_m$.
    We split the approximation error
    \begin{align*}
        \| f - S_m \bm f \|_{ L_2}^{2}
        &= e_2^2
        + \| P(f, V_m,  L_2) - S_m \bm f \|_{ L_2}^{2} \,.
    \end{align*}
    Due to the invariance of $S_m$ to functions in $V_m$, we pull it in front and use compatibility of the operator norm to obtain
    \begin{align*}
        \| f - S_m \bm f \|_{ L_2}^{2}
        &\le e_2^2
        + \|S_m\|_{2\to 2}^2 \sum_{i=1}^{n} \beta(\bm x^i) \lvert(f-P(f, V_m,  L_2))(\bm x^i)\rvert^2 \,.
    \end{align*}
    By \eqref{eq:svd} and the event \eqref{eq:L2woA}, we have $\|S_m\|_{2\to 2}^2 = \|\bm W^{1/2}\bm L\|_{2\to 2}^{-1} \le 2/n$.
    Thus,
    \begin{align*}
        \| f - S_m \bm f \|_{ L_2}^{2}
        &\le 3 e_2^2
        + \frac 2n \sum_{i=1}^{n} \Big\lvert \lvert\omega_i (f-P(f, V_m,  L_2))(\bm x^i)\rvert^2 - e_2^2 \Big\rvert \,.
    \end{align*}

    It remains to estimate the latter summand.
    We define
    \begin{align*}
        \xi_i = \beta(\bm x^i) \Big\lvert(f - P_m f)(\bm x^i)\Big\rvert^2 - e_2^2,
    \end{align*}
    which is mean-zero since we sample with respect to the distribution $\varrho_S$.
    Further, we have
    \begin{align*}
        \mathds E (\xi_i^2)
        = \mathds E\Big((\beta(\bm x^1))^2 \lvert(f-P_m f)(\bm x^i)\rvert^4\Big) - e_2^4
        \le \|f-P_m f\|_{ L_\infty}^{2} e_2^2 - e_2^4
        \le e_2^2 (e_2 + e_\infty)^2 \,,
    \end{align*}
    and
    \begin{align*}
        \|\xi_i\|_{\infty}
        \le \sup_{x\in D}\Big\lvert
            \beta(\bm x) \lvert(f-P_m f)(\bm x)\rvert^2 - e_2^2
        \Big\rvert
        \le
            e_\infty^2
            +e_2^2 \,.
    \end{align*}
    Thus, the conditions in order to apply Bernstein are fulfilled:
    \begin{align}
        \frac 1n\sum_{i=1}^n \xi_i
        &\le 
        \frac{2t}{3n}\Big(e_2^2 + e_\infty^2\Big)
        +\sqrt{\frac{2t}{n}} \Big(e_\infty e_2
        +e_2^2\Big) \nonumber \\
        &\le \Big(\frac 23 +\sqrt 2\Big)e_2^2 
        + \sqrt{\frac{2t}{n}} e_\infty e_2
        + \frac{2t}{3n}e_\infty^2
        \label{eq:L2woB}
    \end{align}
    with probability $1-\exp(-t)$, where $t\le n$ was used in the last inequality.
    Thus,
    \begin{align*}
        \| f - S_m \bm f \|_{ L_2}^{2}
        &\le \Big(\frac{13}{3} +2\sqrt 2\Big)e_2^2 
        + \sqrt{\frac{8t}{n}} e_\infty e_2
        + \frac{4t}{3n}e_\infty^2 \\
        &\le \Big(\frac{13}{3} +2\sqrt 2\Big)
        \Big(e_2
        + \sqrt{\frac{t}{n}}e_\infty\Big)^2 \,.
    \end{align*}
    By union bound we obtain the overall probability exceeding the sum of the probabilities of events given by \eqref{eq:L2woA} and \eqref{eq:L2woB}.
\end{proof} 

Provided $N(V_m)/\|\beta(\cdot) N(V_m, \cdot)\|_\infty$ is finite, Theorem~\ref{L2wo} says that the least squares approximation from a finite-dimensional function space $V_m$ and given the oversampling condition has the same error as the $L_2$-projection up to a multiplicative constant with high probability.
This improves on \cite[Theorem~2.1]{CM17} where the same bound was shown in expectation or bounded by the $L_\infty$-error with high probability.

Next, we prove the central theorem which includes the noisy case.

\begin{proof}[Proof of Theorem~\ref{L2w}] We split the approximation error
    \begin{align*}
        \| f - S_m \bm y \|_{ L_2}^{2}
        &\le e(f, V_m,  L_2)_{ L_2} + 2 \| f - S_m\bm f \|_{ L_2}^{2}
        + 2 \|S_m \bm\varepsilon \|_{ L_2}^{2}
    \end{align*}
    and bound the first two summands as in the proof of Theorem~\ref{L2wo} with the events given by \eqref{eq:L2woA} and \eqref{eq:L2woB}.
    Note, the constant changes from $13/3+2\sqrt 2$ to $23/3+4\sqrt 2 \le 14$.
    Now, we focus on the third summand.
    Applying Corollary~\ref{hanson-wright} gives
    \begin{align}
        &\|S_m \bm\varepsilon \|_{ L_2}^{2}
        = \|(\bm L\herm\bm W\bm L)\inv\bm L\herm\bm W \bm\varepsilon \|_{2}^{2} \nonumber \\
        &\le
        128 \|\bm\varepsilon\|_\infty^2 \|(\bm L\herm\bm W\bm L)\inv\bm L\herm\bm W\|_{2\to 2}^2 t+(8\sqrt 3 \|\bm\varepsilon\|_\infty\sqrt{t\sigma^2} + \sigma^2 )\|(\bm L\herm\bm W\bm L)\inv\bm L\herm\bm W\|_F^2 \label{eq:L2wC}
    \end{align}
    with probability $1-\exp(-t)$.
    Since $\bm L\herm\bm W\bm L \in \mathds C^{m\times m}$, the matrix $(\bm L\herm\bm W\bm L)\inv\bm L\bm W^{1/2}$ has rank at most $m$ and, thus, we use $\|\bm A\|_F^2 \le \rank(\bm A)\|\bm A\|_{2\to 2}^2$ to obtain
    \begin{align*}
        \|(\bm L\herm\bm W\bm L)\inv\bm L\herm\bm W\|_F^2
        \le \|\beta\|_\infty m \|(\bm L\herm\bm W\bm L)\inv\bm L\herm\bm W^{1/2}\|_{2\to 2}^2
        \le \|\beta\|_\infty \frac{2m}{n}
    \end{align*}
    where the last inequality follows from  \eqref{eq:svd} and event \eqref{eq:L2woA}.
    Therefore,
    \begin{align*}
        \|S_m \bm\varepsilon \|_{ L_2}^{2}
        \le 
        \|\beta\|_\infty \Big(
        128 \|\bm\varepsilon\|_\infty^2 \frac 2n t+(8\sqrt 3 \|\bm\varepsilon\|_\infty\sqrt{t\sigma^2} + \sigma^2 )\frac{2m}{n}
        \Big) \,.
    \end{align*}
    By union bound we obtain the overall probability exceeding the sum of the probabilities of the events given by \eqref{eq:L2woA}, \eqref{eq:L2woB}, and \eqref{eq:L2wC}.
\end{proof} 

Next, we prove Theorem~\ref{Linfw} bounding the approximation error of least squares in the $L_\infty$-norm.

\begin{proof}[Proof of Theorem~\ref{Linfw}] For abbreviation, we use $ e_2 = e(f, V_m,  L_\infty)_{ L_2} $ and $ e_\infty = e(f, V_m,  L_\infty)_{ L_\infty}$.
    For any $g = \sum_{k=1}^{m} \langle g, \eta_k \rangle_{ L_2} \eta_k \in V_m$ the Hölder-inequality gives an estimate on the $ L_\infty$-norm in terms of the $ L_2$-norm:
    \begin{align}\label{eq:L2Linfty}
        \|g\|_{ L_\infty}
        = \Big\|\sum_{k=1}^{m} \langle g, \eta_k \rangle_{ L_2} \eta_k \Big\|_{ L_\infty}
        \hspace{-6pt}\le \Big\|\sqrt{\sum_{k=1}^{m} \lvert\langle g, \eta_k \rangle_{ L_2}\rvert^2} \sqrt{ \sum_{k=1}^{m} \lvert\eta_k\rvert^2} \Big\|_{ L_\infty}
        \hspace{-6pt}= \sqrt{N(V_m)} \|g\|_{ L_2} \,.
    \end{align}
    Using this, we reduce the $ L_\infty$-case to the $ L_2$-case which we already covered.
    We split the approximation error
    \begin{align*}
        \| f - S_m \bm y \|_{ L_\infty}
        &\le \| f - P(f, V_m,  L_\infty) \|_{ L_\infty}
        + \| P(f, V_m,  L_\infty) - S_m \bm f \|_{ L_\infty}
        + \|S_m \bm\varepsilon \|_{ L_\infty} \\
        &\le e_\infty
        + \sqrt{N(V_m)} \| P(f, V_m,  L_\infty) - S_m \bm f \|_{ L_2}
        + \sqrt{N(V_m)} \|S_m \bm\varepsilon \|_{ L_2} \,.
    \end{align*}
    Analogously to \eqref{eq:L2woB} we obtain
    \begin{align*}
        \| P(f, V_m,  L_\infty) - S_m \bm f \|_{ L_2}^2
        &\le \Big(\frac 23 +\sqrt 2\Big)e_\infty^2 
        + \sqrt{\frac{2t}{n}} e_\infty e_2
        + \frac{2t}{3n}e_2^2 \\
        &\le \Big(\frac 23+\sqrt 2\Big) \Big( e_\infty+ \sqrt{\frac t n} e_2 \Big)^2 \,,
    \end{align*}
    where the last inequality follows from $t\le n$.
    Thus,
    \begin{align*}
        \| f - S_m \bm y \|_{ L_\infty}
        &\le \Big(1+\sqrt{\frac{4+6\sqrt 2}{3} N(V_m)}\Big)
        \Big( e_\infty+ \sqrt{\frac t n} e_2 \Big)
        + \sqrt{N(V_m)} \|S_m \bm\varepsilon \|_{ L_2} \,.
    \end{align*}
    Using the same bound as in Theorem~\ref{L2w} for $\|S_m \bm\varepsilon \|_{ L_2}$ we obtain the assertion.
\end{proof}  \section{Application on the unit cube}\label{sec:application}
In this section we are interested in approximating functions on the $d$-dimensional unit cube $D = [0,1]^d$ when sample points are drawn uniformly, i.e., with respect to the Lebesgue measure $\mathrm d\bm x$.
The deterministic equivalent to uniform sampling are equispaced points.
When using these for polynomial interpolation, Runge already knew in 1901, that higher degree polynomials lead to oscillatory behaviour towards the border which spoil the approximation error.
Even though, we do not interpolate, we will observe similar behaviour using Legendre and Chebyshev polynomials.
We propose alternative bases, which are stable for large $m = \dim(V_m)$ as well.

Throughout this section we have $ L_2 =  L_2((0,1)^d, \mathrm d\bm x)$ unless stated otherwise.

We consider function spaces to know about the decay of the coefficients.
Note, that for individual functions they may decay faster in contrast to the worst-case setting.
Literature for the worst-case setting can be found in the papers \cite{KU21} for random points with logarithmic oversampling, \cite{NSU21, LT22} for subsampled points with linear oversampling and a logarithmic gap in the error bound
(this was made constructive in \cite{BSU22}), and \cite{DKU23} for subsampled points with linear oversampling and sharp bounds.

\subsection{Sobolev spaces on the unit interval}\label{subs:hs} 

Let $d = 1$, $D = [0,1]$ be the unit interval equipped with the Lebesgue measure $ \mathrm d x$.
In order to get hold on appropriate finite-dimensional function spaces $V_m$ for approximation, we have a look at Sobolev spaces $ H^s =  H^s(0,1)$ with integer smoothness $s\ge 0$.
The inner product of these Hilbert spaces is given by
\begin{align*}
    \langle f, g \rangle_{ H^s}
    = \langle f, g \rangle_{ L_2}
    + \langle f^{(s)}, g^{(s)} \rangle_{ L_2} \,.
\end{align*}
Since $\|f\|_{ L_2}^2 \le \|f\|_{ H^s}^2 = \langle f, f \rangle_{ H^s}$, the embedding operator $\Id\colon  H^s\hookrightarrow L_2$ is compact.
Thus, $W = \Id\herm\circ\Id \colon  H^s\to H^s$ is compact and self-adjoint.
Applying the spectral theorem gives for $f\in H^s$
\begin{align*}
    W(f)
    = \sum_{k=0}^{\infty} \sigma_k \langle f,e_k \rangle_{ H^s} e_k
\end{align*}
where $(\sigma_k)_{k=0}^{\infty}$ is the non-increasing rearrangement of the singular values of $W$ and $(e_k)_{k=0}^{\infty}\subset  H^s$ the corresponding system of eigenfunctions forming an orthonormal basis in $ H^s$.
Since
\begin{align*}
    \langle e_k, e_l \rangle_{ L_2}
    = \langle \Id(e_k), \Id(e_l) \rangle_{ L_2}
    = \langle W (e_k), e_l \rangle_{ H^s}
    = \sigma_k^2 \langle e_k, e_l \rangle_{ H^s}
    = \sigma_k^2 \delta_{k,l}\,,
\end{align*}
the functions $\eta_k = \sigma_k\inv e_k$ form an orthonormal system in $ L_2$.
Setting $V_m = \spn\{\eta_k\}_{k=0}^{m-1}$, we obtain for $ H^s \ni f = \sum_{k=0}^{\infty} \langle f, e_k \rangle_{ H^s} e_k $
\begin{align}\label{eq:R}
    e(f,V_m, L_2)_{ L_2}^2
    = \Big\| \sum_{k=m}^{\infty} \langle f, e_k \rangle_{ H^s} e_k \Big\|_{ L_2}^{2}
    = \sum_{k=m}^{\infty} \Big\lvert \langle f, e_k \rangle_{ H^s} \sigma_k \Big\rvert^2
    \le \|f\|_{ H^s}^2 \sigma_m^2 \,.
\end{align}
Thus, the eigenfunctions corresponding to the largest singular values are a canonical candidate for the approximation space $V_m$.
To put this into concrete terms, in the next two theorems, we give the singular values and eigenfunctions for $ H^1$ and $ H^2$.

\begin{theorem}\label{h1basis} The operator $W = \Id\herm\circ\Id \colon  H^1\to H^1$ has singular values $\sigma_k^2 = \frac{1}{1+\pi^2k^2}$ with corresponding $ L_2$-normalized eigenfunctions
    \begin{align*}
        \eta_k(x)
        = \begin{cases}
            1 & \text{for } k=0 \\
            \sqrt 2\cos(\pi k x) & \text{for } k \ge 1 \,.
        \end{cases}
    \end{align*}
\end{theorem} 

\begin{proof} For $\sigma$ a singular value of $W$ with corresponding eigenfunction $\eta\in  H^1$ and $\varphi\in H^1$ a test function, we have
    \begin{align*}
        \langle \eta, \varphi \rangle_{ L_2}
        = \langle \Id(\eta), \Id(\varphi) \rangle_{ L_2}
        = \langle W(\eta), \varphi \rangle_{ H^1}
        = \sigma^2 \langle \eta, \varphi \rangle_{ H^1}
        = \sigma^2 \Big( \langle \eta, \varphi \rangle_{ L_2} + \langle \eta', \varphi' \rangle_{ L_2} \Big) \,.
    \end{align*}
    Partial differentiation yields
    $ \langle \eta', \varphi' \rangle_{ L_2} = \eta'(1)\varphi(1)-\eta'(0)\varphi(0)-\langle \eta'', \varphi \rangle_{ L_2} $.
    Thus
    \begin{align*}
        \Big\langle \frac{1-\sigma^2}{\sigma^2}\eta + \eta'', \varphi \Big\rangle_{ L_2}
        = \eta'(1)\varphi(1)-\eta'(0)\varphi(0) \,.
    \end{align*}
    Since this has to hold for all test functions $\varphi\in H^1$, we obtain the differential equation
    \begin{align*}
        \frac{1-\sigma^2}{\sigma^2}\eta = -\eta''
        \quad\text{with}\quad
        \eta(0)' = \eta(1)' = 0 \,.
    \end{align*}
    The proposed functions are exactly the ones fulfilling this differential equation.
\end{proof}

To our knowledge, the $ H^1$ basis above was originally introduced in \cite{Krein35} and was already considered in \cite[Lemma~4.1]{WW08} with the same proof technique, in \cite{IN08} as a modified Fourier expansion.
It is further used in \cite{SNC16} as a basis for multivariate approximation in the context of samples along tent-transformed rank-1 lattices, and in \cite{Adcock10, AH11, DDP14, CKNS16, KMNN21}.
The following $H^2$ basis was already posed in \cite[Section 3]{AIN12}, where higher-order Sobolev-spaces are found as well.
The proof of the following theorem is found in Appendix~\ref{app:h2}.

\begin{theorem}\label{h2basis} The operator $W = \Id\herm\circ\Id \colon  H^2\to H^2$ has singular values $\sigma_0^2 = 1$ with corresponding $ L_2$-normalized eigenfunctions
    \begin{align*}
        \eta_0(x) = 1
        \quad\text{and}\quad
        \eta_1(x) = 2\sqrt 3 x-\sqrt 3
    \end{align*}
    and for $k\ge 2$, $\sigma_k^2 = \frac{1}{1+t_k^4}$ with $t_k>0$ the solutions of $\cosh(t_k)\cos(t_k) = 1$ ($t_k \approx \frac{2k-1}{2}\pi$, cf. Lemma~\ref{lemma:coszeros}) and
    \begin{align*}
        \eta_k(x) = \cosh(t_kx)+\cos(t_kx)-\displaystyle{\frac{\cosh(t_k)-\cos(t_k)}{\sinh(t_k)-\sin(t_k)}(\sinh(t_kx)+\sin(t_kx))} \,.
    \end{align*}
    Further, it holds
    \begin{align*}
        \|\eta_k\|_\infty
        \le \begin{cases} 1 & \text{for } k = 0 \\
         \sqrt 3 & \text{for } k = 1 \\
         \sqrt 6 & \text{for } k \ge 2 \,.
        \end{cases}
    \end{align*}
\end{theorem} 

The singular values $\sigma_k$ for $ H^2$ decay quadratic in contrast to linearly for $ H^1$.
Thus, approximating a twice differentiable function, $m = \dim(V_m)$ can be chosen smaller when using the $ H^2$ basis whilst achieving the same truncation error $e(f,V_m,  L_2)_{ L_2}$.
Furthermore, as noise enters with the factor $m/n$, cf.\ Theorem~\ref{L2w}, this helps prevent overfitting as well and leads to a smaller approximation error.

However, as $\cosh$ and $\sinh$ both grow exponentially, the representation of the $ H^2$ basis in Theorem~\ref{h2basis} is prone to cancellations and, therefore, numerical unstable.
In the next theorem we pose an approximation which is numerically stable.

\begin{theorem}\label{h2approximation} For $0<t_2<t_3<\dots$ fulfilling $\cosh(t_k)\cos(t_k) = 1$, let $\eta_k$ be as in Theorem~\ref{h2basis}.
    Further, for $n \ge 2$, let $\tilde t_k = \pi(2k-1)/2$ and
    \begin{align*}
        &\tilde \eta_k(x)
        = \sqrt 2\cos\Big(\tilde t_k x+\pi/4\Big)\\
        &\quad+\mathds 1_{[0,1/2]}(x)\exp\Big(-\tilde t_k x\Big)+\mathds 1_{[1/2,1]}(x)(-1)^k\exp\Big(-\tilde t_k (1-x)\Big) \,.
    \end{align*}
    Then $\lvert \eta_k(x) -\tilde\eta_k(x) \rvert \le \varepsilon $ for $k \ge \frac 2\pi \log(16/\varepsilon)+1$.
    In particular, the approximation $\tilde \eta_k$ is exact up to machine precision $\varepsilon = 10^{-16}$ for $k\ge 27$.
\end{theorem} 

For the proof see Appendix~\ref{app:h2}.
For the numerical experiments we use the exact representation from Theorem~\ref{h2basis} for $m < 10$ and the approximation from Theorem~\ref{h2approximation} for $m\ge 10$.

\subsection{Polynomial approximation on the unit interval} 

Next, we examine how the $H^1$ and $ H^2$ bases compares to polynomial approximation when points are distributed uniformly, i.e., $V_m = \Pi_m = \spn\{1, x, \dots, x^{m-1}\}$ and $\varrho_S(x) = 1/\beta(x) d\varrho_T(x) =  \mathrm d x$.
Polynomial approximation results often assume $f\in X^s$ with
\begin{align*}
    X^s
    \coloneqq \{f\colon[0,1]\to\mathds C : f,\dots,f^{(s-1)} \text{ absolute continuous},\, f^{(s)}\in BV([0,1])\}\,,
\end{align*}
where $BV([0,1])$ are all functions with bounded variation.
This assumtion is stronger than assuming $f\in H^s$ as the following remark shows.

\begin{remark}[$X^s\hookrightarrow H^{s+1/2-\varepsilon}$]\label{crazyfunctionspaces} For a rigorous investigation of the relation of $X^s$ and $H^s$, we need to define the Besov space $B_{p,q}^s$ for $p=1$, $q=\infty$, and integer smoothness $s$
    \begin{align*}
        B_{1,\infty}^{s}
        \coloneqq \Big\{ f\in L_1 : \sup_{h\neq 0} \frac{\|\Delta_h^2f^{(s-1)}\|_{ L_1}}{\lvert h\rvert} < \infty \Big\}
    \end{align*}
    with the finite difference $(\Delta_h f)(x) \coloneqq f(x+h)-f(x)$ and $\Delta_h^2 = \Delta_h\circ\Delta_h$, cf.\ \cite[Section~1.2.5]{Triebel92}.
    
    For $f\in X^s$ the derivative $f^{(s)}$ is of bounded variation.
    Thus, also the finite difference $\Delta_h^2 f$ is of bounded variation.
    In particular, $f^{(s)}\in L_1$ and, therefore, $f\in B_{1,\infty}^{s+1}$.
    By \cite[(2.3.2/23)]{Triebel92}, we further have $B_{1,\infty}^{s+1}\hookrightarrow B_{1,1}^{s+1-\varepsilon}$ for any $\varepsilon > 0$.
    Thus,
    \begin{align*}
        X^s
        \hookrightarrow B_{1,\infty}^{s+1}
        \hookrightarrow B_{1,1}^{s+1-\varepsilon}
\hookrightarrow H^{s+1/2-\varepsilon} \,,
    \end{align*}
    where the third embedding follows from the Sobolev inequality, cf.\ \cite[(2.7.1/1)]{Triebel83}, and the Sobolev space for non-integer smoothness $s$ consists of functions $f$ such that $\langle f, \eta_k\rangle \le Ck^{-s}$ for some constant $C<\infty$.
\end{remark} 

Assuming $f\in X^s$, we have a look into approximating with Legendre- and Chebyshev polynomials:

\begin{itemize}
\item
    The canonical choice for the target measure is $\mathrm d\varrho_T =  \mathrm dx$ and $\beta\equiv 1$ such that $\mathrm \varrho_S(x) = 1/\beta(x) \mathrm d\varrho_T(x) =  \mathrm d x$.
    Orthogonalizing the first $m-1$ monomials with respect to $\mathrm d\varrho_T(x) =  \mathrm dx$, we obtain the Legendre polynomials $P_k$.
    
    For the error of the projection, assuming $f\in  H^s$, the following was shown in \cite[Theorem~3.5]{Wang21}:
    \begin{align}\label{eq:legbound}
        e(f, V_m,  L_2)_{ L_2}^2
        \le \frac{2V}{\pi(s+1/2)(m-s)^{2s+1}} \,,
    \end{align}
    where $V$ is the total variation of $f^{(s)}$.
    With this stronger assumption $f\in X^s$ half an order is gained by polynomial approximation in contrast to the $H^s$ bases.
    This is expected as we require half an order of smoothness in $L_2$ more, cf.\ Remark~\ref{crazyfunctionspaces}. (In the later numerical experiments, we observe the gain of half an order for $ H^1$ and $ H^2$ as well.)

    A drawback comes with the Christoffel function $N(V_m,\cdot)$.
    Since $N(V_m, 0) = m^2$, cf.\ \eqref{eq:stupidlegendre}, this spoils the choice of $m$ to quadratic oversampling:
    \begin{align*}
        10m^2(\log(m)+t) \le n \,,
    \end{align*}
    which is usual for polynomial approximation in uniform points, cf.\ \cite{MNST14}.
\item
    When using the Chebyshev measure $\mathrm d \varrho_T(x) = (1-(2x-1)^2)^{-1/2}\;\mathrm dx$ we have to compensate with $\beta(x) = (1-(2x-1)^2)^{-1/2}$ to obtain uniform random samples as well.
    Orthogonalizing the first $m-1$ monomials with respect to the Chebyshev measure, we obtain the Chebyshev polynomials $T_k$, which are a BOS.
    
    As for the error, assuming $f\in X^s$, we use \cite[Theorem~7.1]{Trefethen13} or \cite[Theorem~6.16]{PPST18} to obtain the same bound as for Legendre polynomials \eqref{eq:legbound} but with respect to the $ L_2((0,1), (1-(2x-1)^2)^{-1/2}\mathrm dx)$ norm.

    As $\|\beta\|_{\infty}$ diverges at the border, this spoils the choice of the polynomial degree $m$ and our bound.
    Note, when we exclude some area around the border for sampling, it does not diverge and the resulting error can be controlled.
    This is called padding and was done in \cite[Section~4.1.2]{PS22}
\end{itemize}

Thus, with polynomial approximation we assume half an order of smoothness more, cf.\ Remark~\ref{crazyfunctionspaces}, which we also see in the approximation rate $\mathcal O(m^{s+1/2})$.

\begin{remark}
    Note, when using the Chebyshev polynomials and samples with respect to the Chebyshev measure, we have $\beta \equiv 1$.
    Since the Chebyshev polynomials are a BOS, this does not spoil our bounds.

    Furthermore, using the Legendre polynomials ($\mathrm d \varrho_T =  \mathrm dx$) and samples with respect to the Chebyshev measure ($\beta(x) = \pi (1-(2x-1)^2)^{1/2}$) works as well.
    To see this we use \cite[Lemma~5.1]{RW12}:
    \begin{align*}
        \sqrt{1-(2x-1)^2} \lvert P_k(x)\rvert^2
        \le \frac{2}{\pi}\Big(2+\frac 1 k\Big)
    \end{align*}
    for $k\ge 1$.
    Thus, $\|\beta(\cdot) N(V_m,\cdot)\|_\infty$ and $\|\beta(\cdot)\|_\infty$ are bounded and do not spoil the choice of polynomial degree $m$ nor the error bound.
\end{remark}

\subsection{Numerics on the unit interval} 

To support our findings, we give a numerical example.
As a test function we use
\begin{align}\label{eq:B2cut}
    f(x) = B_2^{\text{cut}}(x)
    \quad\text{with}\quad
    B_2^{\text{cut}}(x) = \begin{cases}
        -x^2+3/4 &\text{for }x\in [0, 1/2] \\
        x^2/2-3/2x+9/8 &\text{for } x\in[1/2, 1]\\
    \end{cases}
\end{align}
which was already considered in \cite{PV15, NP21}.
The function $B_2^{\text{cut}}$ is shown in Figure~\ref{fig:example_1d} and is a cutout of the B-spline of order two.
It and its first derivative are absolute continuous and the second derivative is of bounded variation.
Therefore $f\in X^3$ and the polynomial approximation bounds from above are applicable.
According to Remark~\ref{crazyfunctionspaces} we further have $f\in H^{5/2-\varepsilon}$ for any $\varepsilon>0$, i.e., there exists $C\ge 0$ such that for $k\ge 0$ it holds $\langle f, \eta_k\rangle_{ L_2} \le Ck^{-5/2+\varepsilon}$.
In particular, $f\in H^2$ and \eqref{eq:R} is applicable for approximating with the $H^1$ and $H^2$ bases.

We sample $f$ in $10\,000$ uniformly random points and add $0.1\%M$ Gaussian noise to obtain $\bm y = \bm f + \bm\varepsilon$, where $M = \max_{x\in[0,1]} f(x)-\min_{x\in[0,1]} f(x) = 5/8$.
For $V_m$ we consider the four choices from above:
The Chebyshev polynomials $V_m = \spn\{T_k\}_{k=0}^{m-1}$ ($\mathrm d\varrho_T(x) = (1-(2x-1)^2)^{-1/2}\;\mathrm dx$ and $\beta = \pi/2$);
the Legendre polynomials $V_m = \spn\{P_k\}_{k=0}^{m-1}$ ($\mathrm d\varrho_T(x) = \mathrm dx$ and $\beta \equiv 1$);
the first $m$ basis functions of $H^1$ from Theorem~\ref{h1basis}, and the the first $m$ basis functions of $H^2$ from Theorems~\ref{h2basis} and \ref{h2approximation} ($\mathrm d\varrho_T(x) = \mathrm dx$ and $\beta \equiv 1$ as well).
For $m = \dim(V_m)$ up to $1\,000$ we do the following:
\begin{enumerate}[(i)]
\item
    Compute the minimal and maximal singular values of $1/\sqrt n \bm W^{1/2}\bm L$, with $\bm W$ and $\bm L$ given in \eqref{eq:lsqrmatrix}.
\item
    We use least squares with $20$ iterations to obtain the approximation $S_m\bm y = \sum_{k=0}^{m-1} \hat g_k \eta_k$, defined in \eqref{eq:lsqrmatrix}.
\item
    We compute the $ L_2$-error by using Parseval's equality:
    \begin{align*}
        \|f-S_m\bm y\|_{ L_2}^2
        = \|f\|_{ L_2}^2-\sum_{k=0}^{m-1}\lvert \hat f_k\rvert^2 + \sum_{k=0}^{m-1} \lvert \hat f_k-\hat g_k\rvert^2\,,
    \end{align*}
    where the coefficients $\hat f_k = \langle f, \eta_k \rangle_{ L_2}$ are computed analyticaly.
\item
    We compute the split approximation error:
    \begin{align*}
        \|f-S_m\bm y\|_{ L_2}^2
        \le 
        2\|f-S_m\bm f\|_{ L_2}^2 + 2\|S_m\bm\varepsilon\|_{ L_2}^2 \,,
    \end{align*}
    where we compute both quantities separately, again, using Parseval's equality.
\end{enumerate}
The results are depicted in Figure~\ref{fig:experiment_1d}.
\begin{figure}
    \centering
    \begingroup
  \makeatletter
  \providecommand\color[2][]{\GenericError{(gnuplot) \space\space\space\@spaces}{Package color not loaded in conjunction with
      terminal option `colourtext'}{See the gnuplot documentation for explanation.}{Either use 'blacktext' in gnuplot or load the package
      color.sty in LaTeX.}\renewcommand\color[2][]{}}\providecommand\includegraphics[2][]{\GenericError{(gnuplot) \space\space\space\@spaces}{Package graphicx or graphics not loaded}{See the gnuplot documentation for explanation.}{The gnuplot epslatex terminal needs graphicx.sty or graphics.sty.}\renewcommand\includegraphics[2][]{}}\providecommand\rotatebox[2]{#2}\@ifundefined{ifGPcolor}{\newif\ifGPcolor
    \GPcolortrue
  }{}\@ifundefined{ifGPblacktext}{\newif\ifGPblacktext
    \GPblacktexttrue
  }{}\let\gplgaddtomacro\g@addto@macro
\gdef\gplbacktext{}\gdef\gplfronttext{}\makeatother
  \ifGPblacktext
\def\colorrgb#1{}\def\colorgray#1{}\else
\ifGPcolor
      \def\colorrgb#1{\color[rgb]{#1}}\def\colorgray#1{\color[gray]{#1}}\expandafter\def\csname LTw\endcsname{\color{white}}\expandafter\def\csname LTb\endcsname{\color{black}}\expandafter\def\csname LTa\endcsname{\color{black}}\expandafter\def\csname LT0\endcsname{\color[rgb]{1,0,0}}\expandafter\def\csname LT1\endcsname{\color[rgb]{0,1,0}}\expandafter\def\csname LT2\endcsname{\color[rgb]{0,0,1}}\expandafter\def\csname LT3\endcsname{\color[rgb]{1,0,1}}\expandafter\def\csname LT4\endcsname{\color[rgb]{0,1,1}}\expandafter\def\csname LT5\endcsname{\color[rgb]{1,1,0}}\expandafter\def\csname LT6\endcsname{\color[rgb]{0,0,0}}\expandafter\def\csname LT7\endcsname{\color[rgb]{1,0.3,0}}\expandafter\def\csname LT8\endcsname{\color[rgb]{0.5,0.5,0.5}}\else
\def\colorrgb#1{\color{black}}\def\colorgray#1{\color[gray]{#1}}\expandafter\def\csname LTw\endcsname{\color{white}}\expandafter\def\csname LTb\endcsname{\color{black}}\expandafter\def\csname LTa\endcsname{\color{black}}\expandafter\def\csname LT0\endcsname{\color{black}}\expandafter\def\csname LT1\endcsname{\color{black}}\expandafter\def\csname LT2\endcsname{\color{black}}\expandafter\def\csname LT3\endcsname{\color{black}}\expandafter\def\csname LT4\endcsname{\color{black}}\expandafter\def\csname LT5\endcsname{\color{black}}\expandafter\def\csname LT6\endcsname{\color{black}}\expandafter\def\csname LT7\endcsname{\color{black}}\expandafter\def\csname LT8\endcsname{\color{black}}\fi
  \fi
    \setlength{\unitlength}{0.0500bp}\ifx\gptboxheight\undefined \newlength{\gptboxheight}\newlength{\gptboxwidth}\newsavebox{\gptboxtext}\fi \setlength{\fboxrule}{0.5pt}\setlength{\fboxsep}{1pt}\definecolor{tbcol}{rgb}{1,1,1}\begin{picture}(9060.00,4520.00)\gplgaddtomacro\gplbacktext{\csname LTb\endcsname \put(582,2914){\makebox(0,0)[r]{\strut{}$10^{-2}$}}\csname LTb\endcsname \put(582,3630){\makebox(0,0)[r]{\strut{}$10^{0}$}}\csname LTb\endcsname \put(1009,2352){\makebox(0,0){\strut{}$10^{1}$}}\csname LTb\endcsname \put(1539,2352){\makebox(0,0){\strut{}$10^{2}$}}\csname LTb\endcsname \put(2069,2352){\makebox(0,0){\strut{}$10^{3}$}}}\gplgaddtomacro\gplfronttext{\csname LTb\endcsname \put(1353,2147){\makebox(0,0){\strut{}$m$}}\csname LTb\endcsname \put(1353,4295){\makebox(0,0){\strut{}Chebyshev}}}\gplgaddtomacro\gplbacktext{\csname LTb\endcsname \put(2842,2914){\makebox(0,0)[r]{\strut{}$10^{-2}$}}\csname LTb\endcsname \put(2842,3630){\makebox(0,0)[r]{\strut{}$10^{0}$}}\csname LTb\endcsname \put(3269,2352){\makebox(0,0){\strut{}$10^{1}$}}\csname LTb\endcsname \put(3799,2352){\makebox(0,0){\strut{}$10^{2}$}}\csname LTb\endcsname \put(4329,2352){\makebox(0,0){\strut{}$10^{3}$}}}\gplgaddtomacro\gplfronttext{\csname LTb\endcsname \put(3613,2147){\makebox(0,0){\strut{}$m$}}\csname LTb\endcsname \put(3613,4295){\makebox(0,0){\strut{}Legendre}}}\gplgaddtomacro\gplbacktext{\csname LTb\endcsname \put(5102,2914){\makebox(0,0)[r]{\strut{}$10^{-2}$}}\csname LTb\endcsname \put(5102,3630){\makebox(0,0)[r]{\strut{}$10^{0}$}}\csname LTb\endcsname \put(5529,2352){\makebox(0,0){\strut{}$10^{1}$}}\csname LTb\endcsname \put(6059,2352){\makebox(0,0){\strut{}$10^{2}$}}\csname LTb\endcsname \put(6589,2352){\makebox(0,0){\strut{}$10^{3}$}}}\gplgaddtomacro\gplfronttext{\csname LTb\endcsname \put(5873,2147){\makebox(0,0){\strut{}$m$}}\csname LTb\endcsname \put(5873,4295){\makebox(0,0){\strut{}$\mathrm H^1$ basis}}}\gplgaddtomacro\gplbacktext{\csname LTb\endcsname \put(7362,2914){\makebox(0,0)[r]{\strut{}$10^{-2}$}}\csname LTb\endcsname \put(7362,3630){\makebox(0,0)[r]{\strut{}$10^{0}$}}\csname LTb\endcsname \put(7789,2352){\makebox(0,0){\strut{}$10^{1}$}}\csname LTb\endcsname \put(8319,2352){\makebox(0,0){\strut{}$10^{2}$}}\csname LTb\endcsname \put(8849,2352){\makebox(0,0){\strut{}$10^{3}$}}}\gplgaddtomacro\gplfronttext{\csname LTb\endcsname \put(8133,2147){\makebox(0,0){\strut{}$m$}}\csname LTb\endcsname \put(8133,4295){\makebox(0,0){\strut{}$\mathrm H^2$ basis}}}\gplgaddtomacro\gplbacktext{\csname LTb\endcsname \put(582,902){\makebox(0,0)[r]{\strut{}$10^{-10}$}}\csname LTb\endcsname \put(582,1422){\makebox(0,0)[r]{\strut{}$10^{-6}$}}\csname LTb\endcsname \put(582,1942){\makebox(0,0)[r]{\strut{}$10^{-2}$}}\csname LTb\endcsname \put(1009,307){\makebox(0,0){\strut{}$10^{1}$}}\csname LTb\endcsname \put(1539,307){\makebox(0,0){\strut{}$10^{2}$}}\csname LTb\endcsname \put(2069,307){\makebox(0,0){\strut{}$10^{3}$}}}\gplgaddtomacro\gplfronttext{\csname LTb\endcsname \put(1353,102){\makebox(0,0){\strut{}$m$}}}\gplgaddtomacro\gplbacktext{\csname LTb\endcsname \put(2842,902){\makebox(0,0)[r]{\strut{}$10^{-10}$}}\csname LTb\endcsname \put(2842,1422){\makebox(0,0)[r]{\strut{}$10^{-6}$}}\csname LTb\endcsname \put(2842,1942){\makebox(0,0)[r]{\strut{}$10^{-2}$}}\csname LTb\endcsname \put(3269,307){\makebox(0,0){\strut{}$10^{1}$}}\csname LTb\endcsname \put(3799,307){\makebox(0,0){\strut{}$10^{2}$}}\csname LTb\endcsname \put(4329,307){\makebox(0,0){\strut{}$10^{3}$}}}\gplgaddtomacro\gplfronttext{\csname LTb\endcsname \put(3613,102){\makebox(0,0){\strut{}$m$}}}\gplgaddtomacro\gplbacktext{\csname LTb\endcsname \put(5102,902){\makebox(0,0)[r]{\strut{}$10^{-10}$}}\csname LTb\endcsname \put(5102,1422){\makebox(0,0)[r]{\strut{}$10^{-6}$}}\csname LTb\endcsname \put(5102,1942){\makebox(0,0)[r]{\strut{}$10^{-2}$}}\csname LTb\endcsname \put(5529,307){\makebox(0,0){\strut{}$10^{1}$}}\csname LTb\endcsname \put(6059,307){\makebox(0,0){\strut{}$10^{2}$}}\csname LTb\endcsname \put(6589,307){\makebox(0,0){\strut{}$10^{3}$}}}\gplgaddtomacro\gplfronttext{\csname LTb\endcsname \put(5873,102){\makebox(0,0){\strut{}$m$}}}\gplgaddtomacro\gplbacktext{\csname LTb\endcsname \put(7362,902){\makebox(0,0)[r]{\strut{}$10^{-10}$}}\csname LTb\endcsname \put(7362,1422){\makebox(0,0)[r]{\strut{}$10^{-6}$}}\csname LTb\endcsname \put(7362,1942){\makebox(0,0)[r]{\strut{}$10^{-2}$}}\csname LTb\endcsname \put(7789,307){\makebox(0,0){\strut{}$10^{1}$}}\csname LTb\endcsname \put(8319,307){\makebox(0,0){\strut{}$10^{2}$}}\csname LTb\endcsname \put(8849,307){\makebox(0,0){\strut{}$10^{3}$}}}\gplgaddtomacro\gplfronttext{\csname LTb\endcsname \put(8133,102){\makebox(0,0){\strut{}$m$}}}\gplbacktext
    \put(0,0){\includegraphics[width={453.00bp},height={226.00bp}]{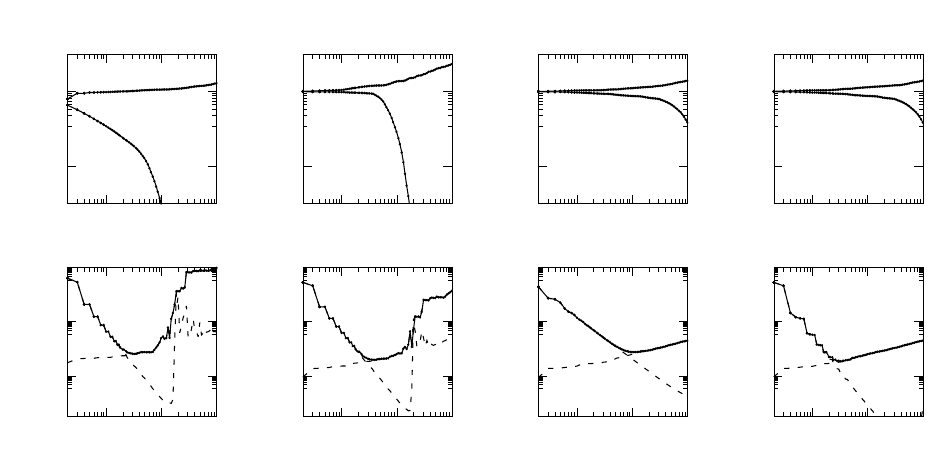}}\gplfronttext
  \end{picture}\endgroup
     \caption{One-dimensional experiment for different choices of $V_m$.
    Top row: minimal and maximal singular value of $1/\sqrt n \bm W^{1/2}\bm L$.
    Bottom row: the $L_2$-approximation error $\|f-S_m\bm y\|_{L_2}^2$ (solid line) split into the error for exact function values $\|f-S_m\bm f\|_{L_2}^2$ and the noise error $\|S_m\bm\varepsilon\|_{L_2}^2$ (dashed lines) with respect to $m$.}\label{fig:experiment_1d}
\end{figure}

\begin{itemize}
\item
    The smallest singular values for the Chebyshev polynomials and the Legendre polynomials decay rapidly for bigger $m$.
    This coincides with the violation of the assumtion in Lemma~\ref{l:trand} for small $m$:
    \begin{align*}
        10\|\beta(\cdot) N(V_m,\cdot)\|_\infty(\log(m)+t) \le n,
    \end{align*}
    where $\|\beta(\cdot) N(V_m,\cdot)\|_\infty$ is unbounded in the Chebyshev case and quadratic in the Legendre case, cf.\ \eqref{eq:stupidlegendre}.
    In this experiment, for $m=1\,000$ the condition number $\sigma_{\max}(\bm W^{1/2}\bm L)/\sigma_{\min}(\bm W^{1/2}\bm L)$ exceeded $10^{29}$ for the algebraic polynomials and was below 14 for the $ H^s$ basis.
\item
    The error for exact function values $\|f-S_m\bm f\|_{ L_2}^2$ has decay $3/2$ for $ H^1$ and $5/2$ for the other bases.
    This conforms with the theory for the polynomial bases.
    For the $ H^1$ and $ H^2$ bases the theory predicted only decay rate $1$ and $2$, cf.\ Theorems~\ref{h1basis}, \ref{h2basis}, and \eqref{eq:R}.
\item
    For the noise error $\|S_m\bm\varepsilon\|_{ L_2}^2$ we observe linear growth in $m = \dim(V_m)$ as predicted in Theorem~\ref{L2w}.
    Furthermore, this error is bigger by a factor of around $40$ in the Chebyshev case compared to the others.
    The maximal weight $\|\bm W\|_\infty$ in this case is around $40$ as well.
    The error due to noise in our bound has the factor $\|\beta\|_\infty$ which can be replaced by $\|\bm W\|_\infty$ to sharpen the bound and explain this effect.
\end{itemize}

This numerical experiment and the earlier theoretical discussion shows, that the $ H^1$ and the $ H^2$ bases are suitable for approximating functions on the unit interval given in uniform random samples.
They are numerically stable in contrast to polynomial approximation with Chebyshev or Legendre.
In particular, the least squares matrix is well-conditioned and we can limit the iterations when using an iterative solver, cf.\ \cite[Theorem~3.1.1]{Greenbaum97}.

\subsection{Sobolev spaces with dominating mixed smoothness on the unit cube} 

The ideas from Subsection~\ref{subs:hs} can be extended to higher dimensions using the concept of dominating mixed smoothness.
We focus on the case of $ H^1$ and $ H^2$, but the same can be done for polynomials as well, cf.\ \cite[Section~8.5.1]{STW11}.

Let $D = [0,1]^d$ be the $d$-dimensional unit cube equipped with the Lebesgue measure $ \mathrm d\bm x$.
The Sobolev space with dominating mixed smoothness of integer degree $s\ge 0$ is given by $\Hmix^s = \Hmix^s(0,1)^d =  H^s(0,1) \otimes\dots\otimes H^s(0,1)$.
The inner product of these Hilbert spaces is given by
\begin{align*}
    \langle f, g \rangle_{\Hmix^s}
    = \sum_{\bm j\in \{0, s\}^d} \langle D^{(\bm j)} f, D^{(\bm j)} g \rangle_{ L_2} \,.
\end{align*}
With $\sigma_k$ and $\eta_k$ the singular values and eigenfunctions of $ H^s$, the singular values and eigenfunctions of $W = \Id\herm\circ\Id \colon \Hmix^s \to \Hmix^s$ extend as follows:
\begin{align*}
    \sigma_{\bm k}^2 = \prod_{j=1}^{d} \sigma_{k_j}^2
    \quad\text{and}\quad
    \eta_{\bm k}(\bm x) = \prod_{j=1}^{d} \eta_{k_j}(x_j) \,.
\end{align*}
To obtain the eigenfunctions corresponding to the smallest singular values, we now work with multi-indices $\bm k$.
The indices corresponding to the largest singular values lie on a, so called, hyperbolic cross
\begin{align*}
    I_R(\Hmix^s)
    \coloneqq \Big\{ \bm k\in\mathds N^d : \prod_{j=1}^{d} \sigma_{k_j}^2 \ge R \Big\} \,.
\end{align*}
For $V_m = \spn\{\eta_{\bm k} : \bm k\in I_R(\Hmix^s)\}$ and $f\in\Hmix^s$, we obtain by \eqref{eq:R}
\begin{align*}
    e(f,V_m, L_2)_{ L_2}
    \le R\|f\|_{\Hmix^s}^2\,.
\end{align*}
In Figure~\ref{fig:hc} we have equally sized index sets for $\Hmix^1$ and $\Hmix^2$.
Note, that $R$ is smaller for $\Hmix^2$ as its singular values decay faster, cf.\ Theorems~\ref{h1basis} and \ref{h2basis}.
\begin{figure}
    \centering
    \begingroup
  \makeatletter
  \providecommand\color[2][]{\GenericError{(gnuplot) \space\space\space\@spaces}{Package color not loaded in conjunction with
      terminal option `colourtext'}{See the gnuplot documentation for explanation.}{Either use 'blacktext' in gnuplot or load the package
      color.sty in LaTeX.}\renewcommand\color[2][]{}}\providecommand\includegraphics[2][]{\GenericError{(gnuplot) \space\space\space\@spaces}{Package graphicx or graphics not loaded}{See the gnuplot documentation for explanation.}{The gnuplot epslatex terminal needs graphicx.sty or graphics.sty.}\renewcommand\includegraphics[2][]{}}\providecommand\rotatebox[2]{#2}\@ifundefined{ifGPcolor}{\newif\ifGPcolor
    \GPcolortrue
  }{}\@ifundefined{ifGPblacktext}{\newif\ifGPblacktext
    \GPblacktexttrue
  }{}\let\gplgaddtomacro\g@addto@macro
\gdef\gplbacktext{}\gdef\gplfronttext{}\makeatother
  \ifGPblacktext
\def\colorrgb#1{}\def\colorgray#1{}\else
\ifGPcolor
      \def\colorrgb#1{\color[rgb]{#1}}\def\colorgray#1{\color[gray]{#1}}\expandafter\def\csname LTw\endcsname{\color{white}}\expandafter\def\csname LTb\endcsname{\color{black}}\expandafter\def\csname LTa\endcsname{\color{black}}\expandafter\def\csname LT0\endcsname{\color[rgb]{1,0,0}}\expandafter\def\csname LT1\endcsname{\color[rgb]{0,1,0}}\expandafter\def\csname LT2\endcsname{\color[rgb]{0,0,1}}\expandafter\def\csname LT3\endcsname{\color[rgb]{1,0,1}}\expandafter\def\csname LT4\endcsname{\color[rgb]{0,1,1}}\expandafter\def\csname LT5\endcsname{\color[rgb]{1,1,0}}\expandafter\def\csname LT6\endcsname{\color[rgb]{0,0,0}}\expandafter\def\csname LT7\endcsname{\color[rgb]{1,0.3,0}}\expandafter\def\csname LT8\endcsname{\color[rgb]{0.5,0.5,0.5}}\else
\def\colorrgb#1{\color{black}}\def\colorgray#1{\color[gray]{#1}}\expandafter\def\csname LTw\endcsname{\color{white}}\expandafter\def\csname LTb\endcsname{\color{black}}\expandafter\def\csname LTa\endcsname{\color{black}}\expandafter\def\csname LT0\endcsname{\color{black}}\expandafter\def\csname LT1\endcsname{\color{black}}\expandafter\def\csname LT2\endcsname{\color{black}}\expandafter\def\csname LT3\endcsname{\color{black}}\expandafter\def\csname LT4\endcsname{\color{black}}\expandafter\def\csname LT5\endcsname{\color{black}}\expandafter\def\csname LT6\endcsname{\color{black}}\expandafter\def\csname LT7\endcsname{\color{black}}\expandafter\def\csname LT8\endcsname{\color{black}}\fi
  \fi
    \setlength{\unitlength}{0.0500bp}\ifx\gptboxheight\undefined \newlength{\gptboxheight}\newlength{\gptboxwidth}\newsavebox{\gptboxtext}\fi \setlength{\fboxrule}{0.5pt}\setlength{\fboxsep}{1pt}\definecolor{tbcol}{rgb}{1,1,1}\begin{picture}(7080.00,2820.00)\gplgaddtomacro\gplbacktext{}\gplgaddtomacro\gplfronttext{\csname LTb\endcsname \put(1764,2493){\makebox(0,0){\strut{}\shortstack{$I_R(\Hmix^1)$ with $254$ frequencies\\ and $R = 5.3\cdot 10^{-5}$}}}}\gplgaddtomacro\gplbacktext{}\gplgaddtomacro\gplfronttext{\csname LTb\endcsname \put(5294,2493){\makebox(0,0){\strut{}\shortstack{$I_R(\Hmix^2)$ with $254$ frequencies\\and $R = 8.3\cdot 10^{-8}$}}}}\gplbacktext
    \put(0,0){\includegraphics[width={354.00bp},height={141.00bp}]{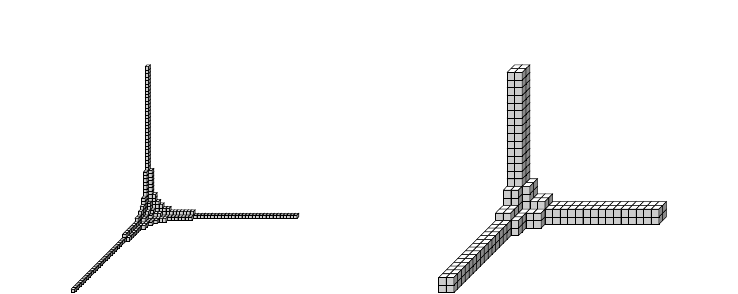}}\gplfronttext
  \end{picture}\endgroup
     \caption{Hyperbolic cross in three dimensions.}\label{fig:hc}
\end{figure}

\subsection{Numerics on the unit cube} 

For a numerical experiment we do the same as in the one-dimensional case but only consider the $\Hmix^2$ case.
For our test function we tensorize the B-Spline cutout
\begin{align*}
    f(\bm x) = \prod_{j=1}^{d} B_2^{\text{cut}}(x_j)
\end{align*}
where $B_2^{\text{cut}}$ was defined in \eqref{eq:B2cut}.
We increase the dimension to $d=5$ and the number of samples to $1\,000\,000$ and use Gaussian noise with variance $\sigma^2 \in \{0.00,\, 0.01M,\, 0.03M\}$ where $M = \max_{\bm x\in[0,1]^d} f(\bm x)-\min_{\bm x\in[0,1]^d} f(\bm x) = 5/8$.

Let $V_m = \spn\{\eta_{\bm k} : \bm k\in I_R(\Hmix^2)\}$ of size $m$ with $\eta_{\bm k}$ the tensorized $\Hmix^2$ basis, cf.\ Theorems~\ref{h2basis} and \ref{h2approximation}.
Since the $\Hmix^2$ basis is a BOS, we obtain
\begin{align*}
    \frac{N(V_m)}{m} \le 6 \,.
\end{align*}
With $t=6$, we satisfy the assumptions of Theorem~\ref{L2w} for $m\le 12\,250$ and obtain a probability exceeding $0.99$ for the error bound in Theorem~\ref{L2w}.
For $m = \dim(V_m)$ up to $10\,000$ we do the following:
\begin{enumerate}[(i)]
\item
    We use plain least squares with $20$ iterations to obtain the approximation $S_m\bm y = \sum_{k=0}^{m-1} \hat g_k \eta_k$, defined in \eqref{eq:lsqrmatrix}.
\item
    We compute the $ L_2$-error by using Parseval's equality analog to the one-dimensional case.
\item
    We compute our bound:
    Applying Theorem~\ref{L2w} using \eqref{eq:L2Linfty}, $t=6$, and $n = 1\,000\,000$, we obtain
    \begin{align*}
        \| f - S_m \bm y \|_{ L_2}^{2}
        &\le 14 \Big(1+\sqrt{\frac{6 N(V_m)}{n}}\Big)^2 e(f,V_m, L_2)_{ L_2}^2 \\
        &\quad+ 
        \frac{m}{n}
        \Big( 138 B \sqrt{\sigma^2} + 4\sigma^2 \Big)
        +
        0.0031 B^2
    \end{align*}
    with probability exceeding $0.99$ and all the remaining quantities known in our experiment.
\end{enumerate}

\begin{figure}
    \centering
    \begingroup
  \makeatletter
  \providecommand\color[2][]{\GenericError{(gnuplot) \space\space\space\@spaces}{Package color not loaded in conjunction with
      terminal option `colourtext'}{See the gnuplot documentation for explanation.}{Either use 'blacktext' in gnuplot or load the package
      color.sty in LaTeX.}\renewcommand\color[2][]{}}\providecommand\includegraphics[2][]{\GenericError{(gnuplot) \space\space\space\@spaces}{Package graphicx or graphics not loaded}{See the gnuplot documentation for explanation.}{The gnuplot epslatex terminal needs graphicx.sty or graphics.sty.}\renewcommand\includegraphics[2][]{}}\providecommand\rotatebox[2]{#2}\@ifundefined{ifGPcolor}{\newif\ifGPcolor
    \GPcolortrue
  }{}\@ifundefined{ifGPblacktext}{\newif\ifGPblacktext
    \GPblacktexttrue
  }{}\let\gplgaddtomacro\g@addto@macro
\gdef\gplbacktext{}\gdef\gplfronttext{}\makeatother
  \ifGPblacktext
\def\colorrgb#1{}\def\colorgray#1{}\else
\ifGPcolor
      \def\colorrgb#1{\color[rgb]{#1}}\def\colorgray#1{\color[gray]{#1}}\expandafter\def\csname LTw\endcsname{\color{white}}\expandafter\def\csname LTb\endcsname{\color{black}}\expandafter\def\csname LTa\endcsname{\color{black}}\expandafter\def\csname LT0\endcsname{\color[rgb]{1,0,0}}\expandafter\def\csname LT1\endcsname{\color[rgb]{0,1,0}}\expandafter\def\csname LT2\endcsname{\color[rgb]{0,0,1}}\expandafter\def\csname LT3\endcsname{\color[rgb]{1,0,1}}\expandafter\def\csname LT4\endcsname{\color[rgb]{0,1,1}}\expandafter\def\csname LT5\endcsname{\color[rgb]{1,1,0}}\expandafter\def\csname LT6\endcsname{\color[rgb]{0,0,0}}\expandafter\def\csname LT7\endcsname{\color[rgb]{1,0.3,0}}\expandafter\def\csname LT8\endcsname{\color[rgb]{0.5,0.5,0.5}}\else
\def\colorrgb#1{\color{black}}\def\colorgray#1{\color[gray]{#1}}\expandafter\def\csname LTw\endcsname{\color{white}}\expandafter\def\csname LTb\endcsname{\color{black}}\expandafter\def\csname LTa\endcsname{\color{black}}\expandafter\def\csname LT0\endcsname{\color{black}}\expandafter\def\csname LT1\endcsname{\color{black}}\expandafter\def\csname LT2\endcsname{\color{black}}\expandafter\def\csname LT3\endcsname{\color{black}}\expandafter\def\csname LT4\endcsname{\color{black}}\expandafter\def\csname LT5\endcsname{\color{black}}\expandafter\def\csname LT6\endcsname{\color{black}}\expandafter\def\csname LT7\endcsname{\color{black}}\expandafter\def\csname LT8\endcsname{\color{black}}\fi
  \fi
    \setlength{\unitlength}{0.0500bp}\ifx\gptboxheight\undefined \newlength{\gptboxheight}\newlength{\gptboxwidth}\newsavebox{\gptboxtext}\fi \setlength{\fboxrule}{0.5pt}\setlength{\fboxsep}{1pt}\definecolor{tbcol}{rgb}{1,1,1}\begin{picture}(6800.00,2820.00)\gplgaddtomacro\gplbacktext{\csname LTb\endcsname \put(530,871){\makebox(0,0)[r]{\strut{}$10^{-6}$}}\csname LTb\endcsname \put(530,1310){\makebox(0,0)[r]{\strut{}$10^{-4}$}}\csname LTb\endcsname \put(530,1748){\makebox(0,0)[r]{\strut{}$10^{-2}$}}\csname LTb\endcsname \put(530,2187){\makebox(0,0)[r]{\strut{}$10^{0}$}}\csname LTb\endcsname \put(890,448){\makebox(0,0){\strut{}$10^{2}$}}\csname LTb\endcsname \put(1506,448){\makebox(0,0){\strut{}$10^{3}$}}\csname LTb\endcsname \put(2121,448){\makebox(0,0){\strut{}$10^{4}$}}}\gplgaddtomacro\gplfronttext{\csname LTb\endcsname \put(1353,142){\makebox(0,0){\strut{}$m$}}\csname LTb\endcsname \put(1353,2493){\makebox(0,0){\strut{}$\sigma^2=0.00$}}}\gplgaddtomacro\gplbacktext{\csname LTb\endcsname \put(2790,959){\makebox(0,0)[r]{\strut{}$10^{-4}$}}\csname LTb\endcsname \put(2790,1573){\makebox(0,0)[r]{\strut{}$10^{-2}$}}\csname LTb\endcsname \put(2790,2187){\makebox(0,0)[r]{\strut{}$10^{0}$}}\csname LTb\endcsname \put(3150,448){\makebox(0,0){\strut{}$10^{2}$}}\csname LTb\endcsname \put(3766,448){\makebox(0,0){\strut{}$10^{3}$}}\csname LTb\endcsname \put(4381,448){\makebox(0,0){\strut{}$10^{4}$}}}\gplgaddtomacro\gplfronttext{\csname LTb\endcsname \put(3613,142){\makebox(0,0){\strut{}$m$}}\csname LTb\endcsname \put(3613,2493){\makebox(0,0){\strut{}$\sigma^2=0.01M$}}}\gplgaddtomacro\gplbacktext{\csname LTb\endcsname \put(5050,652){\makebox(0,0)[r]{\strut{}$10^{-4}$}}\csname LTb\endcsname \put(5050,1266){\makebox(0,0)[r]{\strut{}$10^{-2}$}}\csname LTb\endcsname \put(5050,1880){\makebox(0,0)[r]{\strut{}$10^{0}$}}\csname LTb\endcsname \put(5410,448){\makebox(0,0){\strut{}$10^{2}$}}\csname LTb\endcsname \put(6026,448){\makebox(0,0){\strut{}$10^{3}$}}\csname LTb\endcsname \put(6641,448){\makebox(0,0){\strut{}$10^{4}$}}}\gplgaddtomacro\gplfronttext{\csname LTb\endcsname \put(5873,142){\makebox(0,0){\strut{}$m$}}\csname LTb\endcsname \put(5873,2493){\makebox(0,0){\strut{}$\sigma^2=0.03M$}}}\gplbacktext
    \put(0,0){\includegraphics[width={340.00bp},height={141.00bp}]{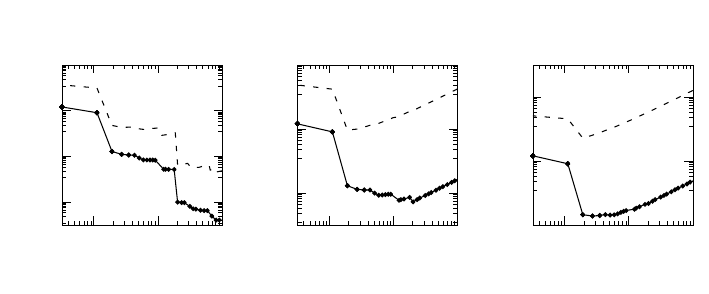}}\gplfronttext
  \end{picture}\endgroup
     \caption{Five-dimensional experiment for $\Hmix^2$. The solid lines represent the $ L_2$-error $\|f-S_m\bm y\|_{ L_2}^2$ and the dashed lines the bound from Theorems~\ref{L2wo} and \ref{L2w}.}
\end{figure}
The results are depicted in Figure~\ref{h2approximation}.
\begin{itemize}
\item 
    The bounds capture the error behaviour well.
    But it seems that there is room for improvement in the constants, especially in the experiments with noise.
    Here, improving constants in the Hanson-Wright inequality in Theorem~\ref{hanson-wright} could be a starting point.
\item
    Furthermor, this experiment shows, that the $ H^2$ basis is easily suitable for high-dimensional approximation as well.
\end{itemize}
 
\section*{Acknowledgement}
I would like to thank Tino~Ullrich for many helpful discussions and his expertise with function spaces.
Further, I would like to thank Michael~Schmischke for several insights on approximation on the unit cube, Sergei~Pereverzyev for his knowledge on regularization theory as well as Ralf~Hielscher and Daniel~Potts for further discussions.

\appendix
\section{Calculations for the $H^2(0,1)$ basis}\label{app:h2}
\begin{proof}[Proof of the first part of Theorem~\ref{h2basis}] Analogously to Theorem~\ref{h1basis}, for $\sigma$ a singular value of $W$ with corresponding eigenfunction $\eta\in \Hmix^2$, we obtain the following differential equation
    \begin{align*}
        \frac{1-\sigma^2}{\sigma^2}\eta = \eta^{(4)}
        \quad\text{with}\quad
        \eta^{(2)}(0)
        = \eta^{(2)}(1)
        = \eta^{(3)}(0)
        = \eta^{(3)}(1)
        = 0.
    \end{align*}

    Now we distinguish three cases for the value of $\sigma^2$: 

    \textbf{First case.}
    Let us assume $\sigma^2 = 1$.
    The ansatz function becomes
        \begin{align*}
        \eta(x)
        = A + Bx + Cx^2 + Dx^3.
        \end{align*}
    From the conditions $\eta^{(2)}(0) = \eta^{(3)}(0) = 0$ we obtain $D = C = 0$.
    The two remaining degrees of freedom are restricted by demanding $ L_2(0,1)$-orthonormality.
    By simple calculus we obtain the proposed eigenfunctions $\eta_0$ and $\eta_1$.
    
    \textbf{Second case.}
    Lets assume $(1-\sigma^2)/\sigma^2 > 0 \Leftrightarrow \sigma^2 < 1$.
    Introducing $t \coloneqq \sqrt[4]{(1-\sigma^2)/\sigma^2}$, we use the ansatz
    \begin{align*}
        \eta(x)
        = A\cos(tx)+B\sin(tx)+C\cosh(tx)+D\sinh(tx).
    \end{align*}
The conditions $\eta^{(2)}(0) = \eta^{(3)}(0) = 0$ transform to $A = C$ and $B = D$, respectively.
    The conditions $\eta^{(2)}(1) = \eta^{(3)}(1) = 0$ can be put into a system of equations:
    \begin{align*}
        \begin{bmatrix} \cosh(t)-\cos(t) & \sinh(t)-\sin(t) \\ \sinh(t)+\sin(t) & \cosh(t)-\cos(t) \end{bmatrix}
        \begin{bmatrix} A \\ B \end{bmatrix}
         = \bm 0
    \end{align*}
    or, by using $\cosh^2(t)-\sinh^2(t) = \cos^2(t)+\sin^2(t) = 1$, equivalently
    \begin{align*}
        \begin{bmatrix} \cosh(t)-\cos(t) & \sinh(t)-\sin(t) \\ 0 & 1-\cosh(t)\cos(t) \end{bmatrix}
        \begin{bmatrix} A \\ B \end{bmatrix}
         = \bm 0.
    \end{align*}
    For non-trivial solutions we need non-regularity of that matrix which transforms to the condition
    $ \cosh(t)\cos(t) = 1 $.
    With the leftover degree of freedom we choose
    \begin{align*}
        A = C = 1
        \quad\text{and}\quad
        B = D = -\frac{\cosh(t)-\cos(t)}{\sinh(t)-\sin(t)}
    \end{align*}
    and obtain $\eta_k$ for $k\ge 2$ as proposed in the theorem.

    For the $ L_2$-norm we obtain
    \begin{align*}
        &\int_0^1 \lvert \eta_n\rvert^2 \; dx
        \quad= \int_0^1 (\cosh(tx)+\cos(tx))^2 \; dx
        + B^2 \int_0^1 (\sinh(tx)+\sin(tx))^2 \; dx \\
        &\qquad+ 2B \int_0^1 (\cosh(tx)+\cos(tx))(\sinh(tx)+\sin(tx)) \; dx \\
        &\quad= 1+ \frac{\sin(2t)+\sinh(2t)+4\cos(t)\sinh(t)+4\sin(t)\cosh(t)}{4t} \\
        &\qquad +B^2 \frac{-\sin(2t)+\sinh(2t)-4\cos(t)\sinh(t)+4\sin(t)\cosh(t)}{4t}
        + 2B\frac{(\sin(t)+\sinh(t))^2}{2t} \\
        &\quad= 1+\frac{1+B^2}{4t}(\sinh(2t)+4\sin(t)\cosh(t)) + \frac{1-B^2}{4t}(\sin(2t)+4\cos(t)\sinh(t)) \\
        &\qquad +B\frac{(\sin(t)+\sinh(t))^2}{t} \\
        &\quad= 1+\frac{1+B^2}{2t}\cosh(t)(\sinh(t)+2\sin(t)) + \frac{1-B^2}{2t}\cos(t)(\sin(t)+2\sinh(t)) \\
        &\qquad +B\frac{(\sin(t)+\sinh(t))^2}{t} \,.
    \end{align*}
    Using $\cos(t)\cosh(t) = 1$, we obtain
    \begin{align}\label{eq:1B}
        1+B^2 = 2\frac{\sinh(t)}{\sinh(t)-\sin(t)}
        \quad\text{and}\quad
        1-B^2 = -2\frac{\sin(t)}{\sinh(t)-\sin(t)} \,.
    \end{align}
    Thus,
    \begin{align*}
        \int_0^1 \lvert \eta_n\rvert^2 \; dx
        &= 1+\frac{1}{t(\sinh(t)-\sin(t))} \Bigg[ \sinh(t)\cosh(t)(\sinh(t)+2\sin(t)) \\
        &\quad- \sin(t)\cos(t)(\sin(t)+2\sinh(t))
        -(\cosh(t)-\cos(t))(\sin(t)+\sinh(t))^2 \Bigg] \\
        &= 1+\frac{\cos(t)\sinh^2(t)-\cosh(t)\sin^2(t)}{t(\sinh(t)-\sin(t))} \\
        &= 1+\frac{\cos(t)\cosh^2(t)-\cos(t)-\cosh(t)+\cosh(t)\cos^2(t)}{t(\sinh(t)-\sin(t))}
    \end{align*}
    where $\cos^2(t)+\sin^2(t) = \cosh^2(t)-\sinh^2(t) = 1$ was used in the last equality.
    Using $\cosh(t)\cos(t) = 1$, the latter summand evaluates to zero and we have proven the $ L_2$-normality.

    \textbf{Third case.}
    Assume $\sigma^2 > 1$. Set $t \coloneqq \sqrt[4]{(\sigma^2-1)/(\sigma^2)}$. The ansatz becomes
    \begin{align*}
        \eta(x)
        &= A\cosh(tx)\cos(tx)
        + B\cosh(tx)\sin(tx)\\
        &\phantom= + C\sinh(tx)\cos(tx)
        + D\sinh(tx)\sin(tx).
    \end{align*}
    The conditions $\eta^{(2)}(0) = \eta^{(3)}(0) = 0$ transform to $D=0$ and $B=C$.
    The two remaining degrees of freedom are fixed by the conditions $\eta^{(2)}(1) = \eta^{(3)}(1) = 0$ which, in matrix form, look as follows
    \begin{align*}
        \begin{bmatrix}
            -\sinh(t)\sin(t) & \sinh(t)\cos(t)-\cosh(t)\sin(t) \\
            -\sinh(t)\cos(t)-\cosh(t)\sin(t) & -2\sinh(t)\sin(t)
        \end{bmatrix}
        \begin{bmatrix} A \\ B \end{bmatrix}
        = \begin{bmatrix}0\\0\end{bmatrix}.
    \end{align*}
    For a non-trivial solution we need that matrix to be non-regular.
    To achieve that we have a look at the roots of its determinant:
    \begin{align*}
        2\sinh^2(t)\sin^2(t)
        +\sinh^2(t)\cos^2(t)-\cosh^2(t)\sin^2(t)
        \overset != 0.
    \end{align*}
        Using $\sin^2(t)+\cos^2(t) = \cosh^2(t)-\sinh^2(t) = 1$ we have
        \begin{align*}
        \sinh^2(t)-\sin^2(t)
        = \frac 12 \cosh(2t) + \frac 12 \cos(2t) - 1
        \overset != 0
        \end{align*}
    which is only fulfilled for $t = 0$, or equivalently, $\sigma^2 = 1$.
    Hence, there are no eigenvalues bigger than 1.
\end{proof}

\begin{lemma}\label{lemma:coszeros} For $0<t_2<t_3<\dots$ fulfilling $\cosh(t_k)\cos(t_k) = 1$ and $\tilde t_k = \frac{2k-1}{2}\pi$, we have
    \begin{align*}
        \frac 32 \pi < t_2
        \quad\text{and}\quad
        \Big\lvert \tilde t_k - t_k \Big\rvert
        \le \varepsilon
    \end{align*}
    for $k \ge \frac 1\pi \log(\pi/\varepsilon)$.
    In particular $ \lvert \tilde t_k - t_k \rvert \le \pi\exp(-2\pi) $ for all $k\ge 2$.
\end{lemma} 

\begin{proof} \begin{figure}
        \centering
        \begingroup
  \makeatletter
  \providecommand\color[2][]{\GenericError{(gnuplot) \space\space\space\@spaces}{Package color not loaded in conjunction with
      terminal option `colourtext'}{See the gnuplot documentation for explanation.}{Either use 'blacktext' in gnuplot or load the package
      color.sty in LaTeX.}\renewcommand\color[2][]{}}\providecommand\includegraphics[2][]{\GenericError{(gnuplot) \space\space\space\@spaces}{Package graphicx or graphics not loaded}{See the gnuplot documentation for explanation.}{The gnuplot epslatex terminal needs graphicx.sty or graphics.sty.}\renewcommand\includegraphics[2][]{}}\providecommand\rotatebox[2]{#2}\@ifundefined{ifGPcolor}{\newif\ifGPcolor
    \GPcolortrue
  }{}\@ifundefined{ifGPblacktext}{\newif\ifGPblacktext
    \GPblacktexttrue
  }{}\let\gplgaddtomacro\g@addto@macro
\gdef\gplbacktext{}\gdef\gplfronttext{}\makeatother
  \ifGPblacktext
\def\colorrgb#1{}\def\colorgray#1{}\else
\ifGPcolor
      \def\colorrgb#1{\color[rgb]{#1}}\def\colorgray#1{\color[gray]{#1}}\expandafter\def\csname LTw\endcsname{\color{white}}\expandafter\def\csname LTb\endcsname{\color{black}}\expandafter\def\csname LTa\endcsname{\color{black}}\expandafter\def\csname LT0\endcsname{\color[rgb]{1,0,0}}\expandafter\def\csname LT1\endcsname{\color[rgb]{0,1,0}}\expandafter\def\csname LT2\endcsname{\color[rgb]{0,0,1}}\expandafter\def\csname LT3\endcsname{\color[rgb]{1,0,1}}\expandafter\def\csname LT4\endcsname{\color[rgb]{0,1,1}}\expandafter\def\csname LT5\endcsname{\color[rgb]{1,1,0}}\expandafter\def\csname LT6\endcsname{\color[rgb]{0,0,0}}\expandafter\def\csname LT7\endcsname{\color[rgb]{1,0.3,0}}\expandafter\def\csname LT8\endcsname{\color[rgb]{0.5,0.5,0.5}}\else
\def\colorrgb#1{\color{black}}\def\colorgray#1{\color[gray]{#1}}\expandafter\def\csname LTw\endcsname{\color{white}}\expandafter\def\csname LTb\endcsname{\color{black}}\expandafter\def\csname LTa\endcsname{\color{black}}\expandafter\def\csname LT0\endcsname{\color{black}}\expandafter\def\csname LT1\endcsname{\color{black}}\expandafter\def\csname LT2\endcsname{\color{black}}\expandafter\def\csname LT3\endcsname{\color{black}}\expandafter\def\csname LT4\endcsname{\color{black}}\expandafter\def\csname LT5\endcsname{\color{black}}\expandafter\def\csname LT6\endcsname{\color{black}}\expandafter\def\csname LT7\endcsname{\color{black}}\expandafter\def\csname LT8\endcsname{\color{black}}\fi
  \fi
    \setlength{\unitlength}{0.0500bp}\ifx\gptboxheight\undefined \newlength{\gptboxheight}\newlength{\gptboxwidth}\newsavebox{\gptboxtext}\fi \setlength{\fboxrule}{0.5pt}\setlength{\fboxsep}{1pt}\definecolor{tbcol}{rgb}{1,1,1}\begin{picture}(4520.00,1700.00)\gplgaddtomacro\gplbacktext{\csname LTb\endcsname \put(4260,687){\makebox(0,0)[l]{\strut{}$x$}}\csname LTb\endcsname \put(70,1434){\makebox(0,0)[l]{\strut{}$y$}}\csname LTb\endcsname \put(643,542){\makebox(0,0)[l]{\strut{}$\pi/2$}}\csname LTb\endcsname \put(1914,542){\makebox(0,0)[l]{\strut{}$3\pi/2$}}\csname LTb\endcsname \put(3227,542){\makebox(0,0)[l]{\strut{}$5\pi/2$}}}\gplgaddtomacro\gplfronttext{}\gplbacktext
    \put(0,0){\includegraphics[width={226.00bp},height={85.00bp}]{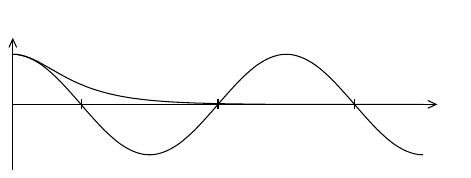}}\gplfronttext
  \end{picture}\endgroup
         \caption{$\cos(t)$ and $1/\cosh(t)$}\label{fig:cos}
    \end{figure}

    Since $0 < 1/\cosh(t) < 1$ for $t > 0$ and the oscillating behaviour of $\cos(t)$, as depicted in Figure~\ref{fig:cos}, we obtain
    \begin{align*}
        t_k \in \begin{cases}
            \Big(\frac{2k-1}{2}\pi, \frac{2k}{2}\pi\Big)
            &\text{for } k \text{ even} \\
            \Big(\frac{2k-2}{2}\pi, \frac{2k-1}{2}\pi\Big)
            &\text{for } k \text{ odd}\,.
        \end{cases}
    \end{align*}
    In particular, $\frac 32 \pi < t_2$.
    Furthermore, for even $k$ and $t \in \Big(\frac{2k-1}{2}\pi, \frac{2k}{2}\pi\Big)$ we have
    \begin{align*}
        \frac{1}{\cosh(t)}
        \le 2\exp(-t)
        \le 2\exp\Big(-\frac{2k-1}{2}\pi\Big)
        \quad\text{and}\quad
        \cos(t)
        \ge \frac{t-\frac{2k-1}{2}\pi}{\pi/2}.
    \end{align*}
    The function bounds intersect for a value larger than $t_k$, which we use to refine the interval:
    \begin{align*}
        t_k
        \in \Big( \tilde t_k ,\tilde t_k + \pi\exp\Big(-\frac{2k-1}{2}\pi\Big) \Big).
    \end{align*}

    Similarly, for odd $k$ and $t \in \Big(\frac{2k-2}{2}\pi, \frac{2k-1}{2}\pi\Big)$ we obtain
    \begin{align*}
        t_k
        \in \Big( \tilde t_k-\pi\exp\Big(-\frac{2k-2}{2}\pi\Big) ,\tilde t_k \Big).
    \end{align*}
    Thus, for $k\ge 2$ we have $\Big\lvert  \tilde t_k-t_k \Big\rvert \le \pi\exp(-(k-1)\pi)$,
    which is smaller than $\varepsilon$ for
    \begin{align*}
        k
        \ge \frac{\log(\pi/\varepsilon)}{\pi} + 1
        \ge \frac{\log(\pi/\varepsilon)}{\pi} \,.
    \end{align*}
\end{proof} 

\begin{lemma}\label{lemma:monotone} For $0<t_2<t_3<\dots$ fulfilling $\cosh(t_k)\cos(t_k) = 1$, we have that $\eta_{t_k}$ defined by
    \begin{align}\label{eq:h2I}
        \eta_{t_k}^{ I}(x) = \cosh(t_kx)-\frac{\cosh(t_k)-\cos(t_k)}{\sinh(t_k)-\sin(t_k)}\sinh(t_kx)
    \end{align}
    is convex and non-negative for all even $k$ and monotone for all odd $k$.
\end{lemma} 

\begin{proof} \textbf{Step 1.} We distinguish for different values of $B = B(t) \coloneqq (\cosh(t)-\cos(t))/(\sinh(t)-\sin(t))$.
    For $B < 1$ we have
        \begin{align*}
        \eta_t^{ I}(x)
        = \cosh(tx)-B\sinh(tx)
        \ge \cosh(tx)-\sinh(tx)
        \ge 0
        \end{align*}
    and by the same argument
    $ (\eta_t^{ I}(x))^{(2)} = t^2\eta_t^{ I}(x) \ge 0 $
    for all $x\ge 0$.
    Thus, $\eta_t^{ I}(x)$ is convex and non-negative.

    For $B > 1$ we obtain
        \begin{align*}
        (\eta_t^{{I}})'
= t(\sinh(tx)-B\cosh(tx))
        \le t(\sinh(tx)-\cosh(tx))
        \le 0
        \end{align*}
    for all $x \ge 0$.
    Thus, $\eta_t^{{I}}$ is monotone.

    \textbf{Step 2.} It is left to show for which $k$'s $B(t_k)$ attains a value smaller or bigger than one:
    \begin{align*}
        B(t_k) \lessgtr 1
        &\quad\Leftrightarrow\quad
        \cosh(t_k)-\cos(t_k) \lessgtr \sinh(t_k)-\sin(t_k) \\
        &\quad\Leftrightarrow\quad
        \exp(-t_k)-\sqrt{2}\cos(t_k+\pi/4) \lessgtr 0.
    \end{align*}
    We will show that $ \exp(-t_k)-\sqrt{2}\cos(t_k+\pi/4) $ has the same sign as $(-1)^{k+1}$ and, thus, are finished.
    We do this by estimating their difference by a quantity smaller than one.
    With $\tilde t_k = \frac{2k-1}{2}\pi$ we obtain
    \begin{align*}
        &\lvert \exp(-t_k)-\sqrt{2}\cos(t_k+\pi/4)-(-1)^{k+1}\rvert \\
&\quad = \lvert \exp(-t_k)-\sqrt{2}\cos(t_k+\pi/4)+\sqrt{2}\cos(\tilde t_k+\pi/4)\rvert.
    \end{align*}
    Using that $\cos$ is Lipschitz-continuous with constant $1$
    and Lemma~\ref{lemma:coszeros} we estimate the above by
    \begin{align*}
        \lvert \exp(-t_k)-\sqrt{2}\cos(t_k+\pi/4)-(-1)^{k+1}\rvert
        &\le \lvert \exp(-t_k)\rvert+\sqrt{2}\lvert t_k-\tilde t_k\rvert \\
        &\le \exp(-3/2\pi)+\sqrt 2 \pi \exp\Big(-2 \pi\Big),
    \end{align*}
    which is certainly smaller than one.
\end{proof}

\begin{lemma}\label{lemma:sym} For $0<t_2<t_3<\dots$ fulfilling $\cosh(t_k)\cos(t_k) = 1$, we have that $\eta_{t_k}^I$ defined in \eqref{eq:h2I} is even with respect to the axis $x=1/2$ for all even $k$ and vice versa.
\end{lemma} 

\begin{proof} \textbf{Step 1.} We will show that $\eta_{t_k}^{ I}$ has any symmetry around $x=1/2$.
    We shift the function and split it into an odd and an even part.
    For $B = B(t) = (\cosh(t)-\cos(t))/(\sinh(t)-\sin(h))$, we obtain
    \begin{align*}
        &\eta_t^{ I}(x+1/2)
        = \cosh(tx+t/2)-B\sinh(tx+t/2) \\
        &\quad= 
        \underbrace{(\cosh(t/2)-B\sinh(t/2))}_{\eqqcolon \alpha}\cosh(tx)
        +\underbrace{(\sinh(t/2)-B\cosh(t/2))}_{\eqqcolon \beta}\sinh(tx).
    \end{align*}
    Multiplying the two factors $\alpha$ and $\beta$ in front of $\cosh(tx)$ and $\sinh(tx)$, we obtain
    \begin{align*}
        \alpha\cdot \beta
        &= -B\cosh^2(t/2)
        -B\sinh^2(t/2)
        +(1+B^2)\cosh(t/2)\sinh(t/2) \\
        &= -B\frac{\cosh(t)-1}{2}
        -B\frac{\cosh(t)+1}{2}
        +(1+B^2) \frac{\sinh(t)}{2} \\
        &= -B\cosh(t)
        +(1+B^2) \frac{\sinh(t)}{2} \,.
    \end{align*}
    Using \eqref{eq:1B}, $\cosh(t)\cos(t) = 1$, and $1 = \cosh^2(t)-\sinh^2(t)$ this evaluates to
    \begin{align*}
        \alpha\cdot \beta
        = -\frac{\cosh^2(t)-1}{\sinh(t)-\sin(t)} + \frac{\sinh^2(t)}{\sinh(t)-\sin(t)}
        = 0 \,.
    \end{align*}
    And since we are not dealing with the zero function either $\alpha$ or $\beta$ is zero.
    Thus, $x\mapsto \eta_t^{ I}(x+1/2)$ obeys a symmetry.

    \textbf{Step 2.}
    It remains to specify the kind of symmetry.
    By Lemma~\ref{lemma:monotone} we have that $\eta_t^{ I}$ is convex for even $k$.
    Since a convex non-constant function cannot be odd it has to be even.
    Also by Lemma~\ref{lemma:monotone} we have that $\eta_t^{ I}$ is monotone for odd $k$.
    Since a monotone non-zero function cannot be even it has to be odd.
\end{proof} 

\begin{proof}[Proof of the second part of Theorem~\ref{h2basis}] The cases $k\in\{0,1\}$ are clear.
    For $k\ge 2$ we split the function into $\eta_t^{ I}$ defined in \eqref{eq:h2I} and
    \begin{align*}
        \eta_t^{{II}}(x) = \cos(tx)-\frac{\cosh(t)-\cos(t)}{\sinh(t)-\sin(t)}\sin(tx).
    \end{align*}
    We will show that each of these is bounded by $1.01\sqrt 2$ and, thus, obtain the assertion.

    \textbf{Step 1.} In order to bound $\eta_{t_k}^{{I}}$ we firstly have a look at the boundary points $x\in\{0, 1\}$.
    With Lemma~\ref{lemma:sym} we obtain
    \begin{equation}\label{eq:I}
        \eta_{t_k}^{{I}}(0) 
        = \Big\lvert  \eta_{t_k}^{ I}(1) \Big\rvert
        = 1.
    \end{equation}
    By Lemma~\ref{lemma:monotone} $\eta_{t_k}^{ I}$ is either non-negative and convex or monotone and, thus, cannot exceed its values on the boundary.

    \textbf{Step 2.} In order to bound $\eta_t^{{II}}$ we define
    \begin{align*}
        B \coloneqq \frac{\cosh(t)-\cos(t)}{\sinh(t)-\sin(t)}
        \quad\text{and}\quad
        \vartheta = \arg(1+B i).
    \end{align*}
    Next, we use the exponential definition of sine and cosine and the polar representation of complex numbers to obtain
    \begin{align*}
        \eta_t^{{II}}(x)
        &= \cos(tx)-B\sin(tx) \\
        &= \frac{\exp( i tx)+\exp(- i tx)}{2}
        +B  i\frac{\exp( i tx)-\exp(- i tx)}{2} \\
        &= \frac{(1+B i)\exp( i tx)+(1-B i)\exp(- i tx)}{2} \\
        &= \sqrt{1+B^2}\frac{\exp( i (tx+\vartheta))+\exp(- i (tx+\vartheta))}{2} \\
        &= \sqrt{1+B^2}\cos(tx+\vartheta)
    \end{align*}
    Thus, by \eqref{eq:1B}
    \begin{align}
        \lvert \eta_t^{{II}}(x)\rvert
        \le \sqrt{1+B^2}
        = \sqrt{\frac{2}{1-\sin(t)/\sinh(t)}}
        \le \sqrt{\frac{2}{1-1/\sinh(t)}} \label{eq:II}
    \end{align}
    From Lemma~\ref{lemma:coszeros} we use $t\ge 3/2\pi$ in combination with the monotonicity in \eqref{eq:II} we have $ \lvert \eta_t^{{II}}(x)\rvert \le 1.01\sqrt 2 $.

\end{proof}

\begin{lemma}\label{lemma:orange} For $t \ge \max\{ 2\log(4/\varepsilon), 3/2\pi \}$ we have for $x\in[0,1/2]$
    \begin{align*}
        \Big\lvert  \Big(1-\frac{\cosh(t)-\cos(t)}{\sinh(t)-\sin(t)}\Big)\sinh(tx) \Big\rvert
        \le \varepsilon \,.
    \end{align*}
\end{lemma} 

\begin{proof} We use $\cosh(t) - \sinh(t) = \exp(-t)$ and $\cos(t)-\sin(t) = \sqrt 2\cos(t+\pi/4)$ to estimate
    \begin{align*}
        \Big\lvert  \Big(1-\frac{\cosh(t)-\cos(t)}{\sinh(t)-\sin(t)}\Big)\sinh(tx) \Big\rvert
        &= \lvert\sqrt 2 \cos(t+\pi/4)-\exp(-t)\rvert\Big\lvert  \frac{\sinh(tx)}{\sinh(t)-\sin(t)} \Big\rvert.
    \end{align*}
    Since $x\le 1/2$, $\sinh$ strictly monotone growing, and $t \ge 3/2\pi$ by Lemma~\ref{lemma:coszeros}, we further estimate
    \begin{align*}
        \Big\lvert  \Big(1-\frac{\cosh(t)-\cos(t)}{\sinh(t)-\sin(t)}\Big)\sinh(tx) \Big\rvert
        &\le 2 \Big\lvert  \frac{\sinh(t/2)}{\sinh(t)-\sin(t)} \Big\rvert \\
        &= 2 \Big\lvert  \frac{1}{2\cosh(t/2)} \frac{1}{1-\sin(t)/\sinh(t)} \Big\rvert
    \end{align*}
    Using $1-\sin(t)/\sinh(t) > 1/2$ for $t > 3/2\pi$, we obtain
    \begin{align*}
        \Big\lvert  \Big(1-\frac{\cosh(t)-\cos(t)}{\sinh(t)-\sin(t)}\Big)\sinh(tx) \Big\rvert
        \le \frac{2}{\cosh(t/2)}
        \le \frac{4}{\exp(t/2)} \,,
    \end{align*}
    which is smaller than $\varepsilon$ for $t \ge 2\log(4/\varepsilon)$.
\end{proof}

\begin{lemma}\label{lemma:small} For $0<t_2<t_3<\dots$ fulfilling $\cosh(t_k)\cos(t_k) = 1$, we have
    \begin{flalign*}
        && \Big\lvert  \eta_{t_k}^{{II}}(x) - \sqrt 2\cos(t_kx+\pi/4) \Big\rvert
        &\le \varepsilon
        \quad\text{for}\quad x\in[0, 1] && \\
        \text{and} && \Big\lvert  \eta_{t_k}^{{I}}(x) - \exp(-tx) \Big\rvert
        &\le \varepsilon
        \quad\text{for}\quad x \in [0, 1/2] &&
    \end{flalign*}
    for $k \ge \frac 2\pi \log(4/\varepsilon) + 1$.
\end{lemma}

\begin{proof}
    \textbf{Step 1.}
    For the first inequality we use
    \begin{align*}
        \sqrt 2\cos(tx+\pi/4)
        = \cos(tx)-\sin(tx)
    \end{align*}
    to obtain
    \begin{align*}
        \Big\lvert  \eta_{t}^{{II}} - \sqrt 2\cos(tx+\pi/4) \Big\rvert
        = \Big\lvert  \Big(1-\frac{\cosh(t)-\cos(t)}{\sinh(t)-\sin(t)}\Big)\sin(tx) \Big\rvert
    \end{align*}
    which is smaller than $\varepsilon$ for $t > \max\{ 2\log(4/\varepsilon), 3/2\pi \}$ by Lemma~\ref{lemma:orange}.
    
    The second inequality follows analogously from $\exp(-tx) = \cosh(tx)-\sinh(tx)$ and Lemma~\ref{lemma:orange}.

    \textbf{Step 2.}
    It is left to show the condition $t \ge \max\{ 2\log(4/\varepsilon), 3/2\pi \}$ from Step 1.
    By Lemma~\ref{lemma:coszeros} we have $t_k \ge 3/2\pi$.
    Further, by assumption, we have
    \begin{align*}
        k \ge \frac 2\pi \log\Big(\frac{4}{\varepsilon}\Big)+1
        \ge \frac{2}{\pi}\log\Big(\frac{4}{\varepsilon}\Big) + \exp(-2\pi) + \frac 12
    \end{align*}
    Thus,
    \begin{align*}
        2\log\Big(\frac{4}{\varepsilon}\Big)
        \le
        \frac{2k-1}{2}\pi - \pi\exp(-2\pi)
        \le t_k
    \end{align*}
    where the last inequality follows from Lemma~\ref{lemma:coszeros}.
\end{proof} 

\begin{proof}[Proof of Theorem~\ref{h2approximation}]
    Because of the symmetry shown in Lemma~\ref{lemma:sym} we assume without loss of generality $x\in[0,1/2]$.
    Then
    \begin{align*}
        &\lvert  \eta_n(x)- \tilde\eta_n(x) \rvert 
        \le 
        \Big\lvert  \eta_n^{ I}(x) - \exp(-t_kx) \Big\rvert
        + \Big\lvert  \eta_n^{{II}}(x) - \sqrt 2\cos(t_kx+\pi/4) \Big\rvert \\
        &\quad + \Big\lvert  \exp(-t_kx) -\exp(-\tilde t_k x) \Big\rvert
        + \Big\lvert  \sqrt 2\cos(t_kx+\pi/4) -\sqrt 2\cos(\tilde t_k x+\pi/4) \Big\rvert.
    \end{align*}
    By Lemma~\ref{lemma:small}, the first two summands are each smaller than $\varepsilon/4$ each for $n > \frac 2\pi \log(16/\varepsilon)+1$.
    We estimate the two latter summands as follows.

    Since $\cos$ is Lipschitz continuous with constant one we have
    \begin{align*}
        \Big\lvert  \sqrt 2\cos(t_k x + \pi/4) - \sqrt 2\cos\Big(\tilde t_k x+\frac{\pi}{4}\Big) \Big\rvert
        \le \Big\lvert  \sqrt 2\Big(t_k-\tilde t_k\Big) \Big\rvert
    \end{align*}
    which, by Lemma~\ref{lemma:coszeros} is smaller than $\varepsilon/4$ for $k > \frac 1\pi \log(4\pi/(\sqrt 2\varepsilon))$.

    Since $\exp$ is Lipschitz continuous with constant $1$ on $(-\infty, 0)$, we have
    \begin{align*}
        \Big\lvert  \exp(-t_k x)
        - \exp(-\tilde t_k x) \Big\rvert
        \le \Big\lvert  t_k - \tilde t_k \Big\rvert
    \end{align*}
    which, by Lemma~\ref{lemma:coszeros} is smaller than $\varepsilon/4$ for $k > \frac 1\pi \log(16/\varepsilon)$.

    Overall, we obtain
    $ \lvert  \eta_n(x)- \tilde\eta_n(x) \rvert < 4\frac{\varepsilon}{4} = \varepsilon $
    for
    \begin{align*}
        k > \max\Big\{
            \frac 2\pi \log(16/\varepsilon)+1,\,
            \frac 1\pi \log(4\pi/(\sqrt 2\varepsilon)),\,
            \frac 1\pi \log(16/\varepsilon)
            \Big\}
        = \frac 2\pi \log(16/\varepsilon)+1\,.
    \end{align*}
\end{proof}
 
\bibliographystyle{plain}

\end{document}